\documentclass[12pt,reqno]{amsart}
\usepackage[utf8]{inputenc}
\usepackage[T1]{fontenc}% Pour pouvoir taper les accent directement et non pas passer par
\usepackage[usenames, dvipsnames]{color}
\usepackage{ulem}

\usepackage{dsfont, amsfonts, amsmath, amssymb,amscd, stmaryrd, latexsym, amsthm, dsfont}
\usepackage[frenchb,english]{babel}
\usepackage{enumerate}
\usepackage{longtable}
\usepackage{geometry}
\usepackage{float}
\usepackage{tikz}
\usetikzlibrary{shapes,arrows}
\usepackage{pifont}
\usepackage{float}
\usepackage{tikz}
\usepackage{xypic}
\usetikzlibrary{shapes,arrows}
\usepackage{multirow}
\usepackage{multicol}
\newtheorem{theorem}{Theorem}[section]
\newtheorem{lemma}[theorem]{Lemma}

\theoremstyle{remark}

\newtheorem{exam}[theorem]{\textbf{Example}}

\usepackage{hyperref}
\hypersetup{
	colorlinks=true,
	urlcolor=blue,
	citecolor=blue}
\def\NN{\mathds{N}}
\def\RR{\mathbb{R}}
\def\QQ{\mathbb{Q}}
\def\CC{\mathbb{C}}
\def\ZZ{\mathbb{Z}}
\def\kk{\mathds{k}}
\usepackage{mathtools}
\usepackage{mleftright}

\begin{document}
	
	\def\NN{\mathbb{N}}
	\def\RR{\mathds{R}}
	\def\HH{I\!\! H}
	\def\QQ{\mathbb{Q}}
	\def\CC{\mathds{C}}
	\def\ZZ{\mathbb{Z}}
	\def\DD{\mathds{D}}
	\def\OO{\mathcal{O}}
	\def\kk{\mathds{k}}
	\def\KK{\mathbb{L}}
	\def\ho{\mathcal{H}_0^{\frac{h(d)}{2}}}
	\def\LL{\mathbb{L}}
	\def\L{\mathds{k}_2^{(2)}}
	\def\M{\mathds{k}_2^{(1)}}
	\def\k{\mathds{k}^{(*)}}
	\def\l{\mathds{L}}

	\selectlanguage{english}
	\title{On the unit group and the $2$-class number of $\mathbb{Q}(\sqrt{2},\sqrt{p},\sqrt{q})$}

	\author[M. M. Chems-Eddin]{Mohamed Mahmoud Chems-Eddin}
	\address{Mohamed Mahmoud CHEMS-EDDIN: Department of Mathematics, Faculty of Sciences Dhar El Mahraz,
		Sidi Mohamed Ben Abdellah University, Fez,  Morocco}
	\email{2m.chemseddin@gmail.com}

	\author[M. B. T. El Hamam]{Moha Ben Taleb El Hamam}
	\address{Moha Ben Taleb El Hamam: Department of Mathematics, Faculty of Sciences Dhar El Mahraz,
		Sidi Mohamed Ben Abdellah University, Fez,  Morocco}
	\email{mohaelhomam@gmail.com}

	 \author[M. A.  Hajjami]{Moulay Ahmed Hajjami} 
	\address{Moulay Ahmed Hajjami, Department of  Mathematics, Faculty of  Sciences and Technology, Moulay Ismail University of Meknes, Errachidia, Morocco.}
	\email{a.hajjami76@gmail.com}


	\begin{abstract}
	Let  $p\equiv 1\pmod{8}$ and $q\equiv7\pmod 8$ be two prime numbers. The purpose of this paper is to compute the unit groups of the fields $\KK=\QQ(\sqrt 2, \sqrt{p}, \sqrt{q} )$ and give their $2$-class numbers. 
	\end{abstract}
	
	\selectlanguage{english}
	
	\maketitle
	%	\tableofcontents
 	\section{\bf Introduction}
	Let  $k$ be a  number field of degree $n$  (i.e., $[k: \QQ] = n$).  Denote by $E_k$  the unit group of $k$ that is  the group of the invertible elements of the ring $\mathcal{O}_k$ of algebraic integers of the number field  $k$.   By the well known Dirichlet's unit theorem, if $n=r+2s$, where $r$ is the number of real embeddings and $s$ the number of conjugate pairs of complex embeddings of $k$, then there exist $t=r+s-1$ units of $\mathcal{O}_k$ that generate $E_k$ (modulo the roots of unity), and these $t$ units are called the {\it  fundamental system of units of $k$}. Therefore, it is well known that
	$$E_k\simeq \mu(k)\times \mathbb{Z}^{r+s-1},$$
	where   $\mu(k)$  is the group of roots of unity contained in $k$.
	
	One major  problem in algebraic number theory (more precisely in theory of units of number fields which is related to almost all areas of algebraic number theory) is the computation of the
	fundamental system of units. For  quadratic fields, the problem
	is easily solved. An early study of unit groups of multiquadratic fields was established by Varmon \cite{Varmon}.    For quartic bicyclic fields, Kubota \cite{Ku-56} gave
	a method for finding a fundamental system of units.
	Wada \cite{wada} generalized Kubota's method, creating an algorithm for computing
	fundamental units in any given multiquadratic field. However, in general, it is not easy to  compute the unit group of a number field  especially for number fields of degree more than $4$. Actually, in  literature  there are only few examples of computation of the unit group of a given number field $k$ of degree $8$ (see   \cite{ChemsUnits9,CAZ,chemszekhniniaziziUnits1}). In the present work, we focus on the computation of the unit group of the real triquadratic fields of the form 
      $\KK=\QQ(\sqrt 2, \sqrt{p}, \sqrt{q} )$, where $p\equiv 1\pmod{8}$ and $q\equiv7\pmod 8$, and furthermore, we give the class number of these fields.

	Notice that the motivation behind the computation of the unit groups of these fields is the fact that $\KK$ is the first layer of the cyclotomic  $\ZZ_2$-extension of the biquadratic field $\QQ(\sqrt{p}, \sqrt{q} )$ (cf. \cite{ChemsUnits9,CAZ}). Furthermore, computing the unit group of the fields  $\KK$ is also the first step to find the unit group of all fields of the form $\KK(\sqrt{-\ell})$, where $\ell\geq 1$ is a positive square-free integer (cf. \cite{azizunints99}). We note that the unit group of these fields are useful for the study of the Hilbert $2$-class field tower of the subfields $\KK(\sqrt{-\ell})$ (see for example \cite{chemszekhniniaziziUnits1}).
	We note also that this paper is a continuation of the paper \cite{Chemsarith} and for further works in the same direction we refer the reader to \cite{Bulant,Varmon,kuvcera1995parity}.

	Let $\varepsilon_\ell$  (resp. $h_2(\ell)$) denote the fundamental unit of (resp. the $2$-class number) of   a real quadratic field $\QQ(\sqrt{\ell})$, where $\ell>1 $ is a positive square-free integer. Let $h_2(k)$ denote the $2$-class number of a number field $k$.
	
	\section{\bf Preliminaries} 	
	Let us start this section by  recalling the method given in    \cite{wada}, that describes a fundamental system  of units of a real  multiquadratic field $K_0$. Let  $\sigma_1$ and 
	$\sigma_2$ be two distinct elements of order $2$ of the Galois group of $K_0/\mathbb{Q}$. Let $K_1$, $K_2$ and $K_3$ be the three subextensions of $K_0$ invariant by  $\sigma_1$,
	$\sigma_2$ and $\sigma_3= \sigma_1\sigma_3$, respectively. Let $\varepsilon$ denote a unit of $K_0$. Then \label{algo wada}
	$$\varepsilon^2=\varepsilon\varepsilon^{\sigma_1}  \varepsilon\varepsilon^{\sigma_2}(\varepsilon^{\sigma_1}\varepsilon^{\sigma_2})^{-1},$$
	and we have, $\varepsilon\varepsilon^{\sigma_1}\in E_{K_1}$, $\varepsilon\varepsilon^{\sigma_2}\in E_{K_2}$  and $\varepsilon^{\sigma_1}\varepsilon^{\sigma_2}\in E_{K_3}$.
	It follows that the unit group of $K_0$  
	is generated by the elements of  $E_{K_1}$, $E_{K_2}$ and $E_{K_3}$, and the square roots of elements of   $E_{K_1}E_{K_2}E_{K_3}$ which are perfect squares in $K_0$.
	
	This method is very useful for computing a fundamental system of units of a real biquadratic number field, however, in the case of real triquadratic 
	number field the problem of the determination of the unit group becomes very difficult and demands some specific computations and eliminations, as what we will see in the next section.
	We shall consider the field $\KK=\mathbb{Q}(\sqrt{2},\sqrt{p},\sqrt{q})$, where $p$ and $q$ are two distinct prime numbers. Thus, we have the following diagram:

	\begin{figure}[H]
		$$\xymatrix@R=0.8cm@C=0.3cm{
			&\KK=\QQ( \sqrt 2, \sqrt{p}, \sqrt{q})\ar@{<-}[d] \ar@{<-}[dr] \ar@{<-}[ld] \\
			k_1=\QQ(\sqrt 2,\sqrt{p})\ar@{<-}[dr]& k_2=\QQ(\sqrt 2, \sqrt{q}) \ar@{<-}[d]& k_3=\QQ(\sqrt 2, \sqrt{pq})\ar@{<-}[ld]\\
			&\QQ(\sqrt 2)}$$
		\caption{Intermediate fields of $\KK/\QQ(\sqrt 2)$}\label{fig:I}
	\end{figure}
	Let $\tau_1$, $\tau_2$ and $\tau_3$ be the elements of  $ \mathrm{Gal}(\KK/\QQ)$ defined by
	\begin{center}	\begin{tabular}{l l l }
			$\tau_1(\sqrt{2})=-\sqrt{2}$, \qquad & $\tau_1(\sqrt{p})=\sqrt{p}$, \qquad & $\tau_1(\sqrt{q})=\sqrt{q},$\\
			$\tau_2(\sqrt{2})=\sqrt{2}$, \qquad & $\tau_2(\sqrt{p})=-\sqrt{p}$, \qquad &  $\tau_2(\sqrt{q})=\sqrt{q},$\\
			$\tau_3(\sqrt{2})=\sqrt{2}$, \qquad &$\tau_3(\sqrt{p})=\sqrt{p}$, \qquad & $\tau_3(\sqrt{q})=-\sqrt{q}.$
		\end{tabular}
	\end{center}
	Note that  $\mathrm{Gal}(\KK/\QQ)=\langle \tau_1, \tau_2, \tau_3\rangle$
	and the subfields  $k_1$, $k_2$ and $k_3$ are
	fixed by  $\langle \tau_3\rangle$, $\langle\tau_2\rangle$ and $\langle\tau_2\tau_3\rangle$ respectively. Therefore,\label{fsu preparations} a fundamental system of units  of $\KK$ consists  of seven  units chosen from those of $k_1$, $k_2$ and $k_3$, and  from the square roots of the elements of $E_{k_1}E_{k_2}E_{k_3}$ which are squares in $\KK$. With these notations, we have:

	\begin{lemma}\label{lm noms esp_2p}
		Let $p\equiv 1\pmod 8$ be a prime number. Put $\varepsilon_{2p}=\beta+\alpha\sqrt{2p}$ with $\beta , \alpha\in\ZZ$. If $N(\varepsilon_{2p})=1$, then  $\sqrt{\varepsilon_{2p}}=\frac{1}{\sqrt2}(\alpha_1+\alpha_2\sqrt{ 2p})$, for some integers $\alpha_1, \alpha_2$ such that $\alpha=\alpha_1 \alpha_2$. It follows that: 
		\begin{eqnarray}\label{T_3_-1_eqi2p_N=1} 
			\begin{tabular}{ |c|c|c|c|c|c|c|c|c|}
				\hline
				$  \sigma $               &$1+\tau_2$         &$1+\tau_1\tau_2$    &$1+\tau_1\tau_3$         &$1+\tau_2\tau_3$&    $1+\tau_1$      \\
				\hline
				$\sqrt{\varepsilon_{2p}}^\sigma$&$(-1)^{u}$           &$-\varepsilon_{2p}$                &	$(-1)^{u+1}$                 &$(-1)^u$  &   $(-1)^{u+1}$                     \\
				\hline
			\end{tabular} 
		\end{eqnarray}
		for some $u$   in $\{0, 1\}$ such that $\frac 12(\alpha_1^2-2p\alpha_2^2)=(-1)^u$\label{the int u}.	
		
		%	********** to add the case N(epq)=-1.****
	\end{lemma}	

To prove this lemma we need to recall the following :
\begin{lemma}[\cite{Az-00},  {Lemma 5}]\label{1:046}
	Let $d>1$ be a square-free integer and $\varepsilon_d=x+y\sqrt d$,
	where $x$,  $y$ are  integers or semi-integers. If $N(\varepsilon_d)=1$,  then $2(x+1)$,  $2(x-1)$,  $2d(x+1)$ and
	$2d(x-1)$ are not squares in  $\QQ$.
\end{lemma}

	\begin{proof}[Proof of Lemma \ref{lm noms esp_2p}]
	As $N(\varepsilon_{2p})=1$, then $\beta^2-1=\alpha^22p$. So by   Lemma \ref{1:046}, we have % ($\beta\pm1= \alpha_1^22$ and  $\beta\mp1= \alpha_2^2p$) or
	 ($\beta\mp1= \alpha_1^2 $  and  $\beta\pm1=  \alpha_2^22p$) for some integers $\alpha_1$ and $\alpha_2$. Thus, $2\beta=  \alpha_1^2+\alpha_2^22p$ and     $\frac 12(\alpha_1^2-2p\alpha_2^2)=(-1)^u$, for some $u$   in $\{0, 1\}$. Therefore, $2\varepsilon_{2p}=2\beta+2\alpha\sqrt{2p}=\alpha_1^22p+\alpha_2^2+2\alpha_1^2+\alpha_2^22p=(\alpha_1+\alpha_2\sqrt{2p})^2$. The reader can deduce the rest easily.
 	\end{proof}

	\begin{lemma}[\cite{azizitalbi}, Theorem  6] \label{units of k1}Let $p\equiv 1\pmod 4$ be a prime number and $k_1=\QQ(\sqrt 2,\sqrt{p})$. We have 
		\begin{enumerate}[\rm 1)]
			\item  If $N(\varepsilon_{2p})=-1$, then $\{\varepsilon_{2}, \varepsilon_{p},	\sqrt{\varepsilon_{2}\varepsilon_{p}\varepsilon_{2p}}\}$ is a fundamental system of units
			of $k_1$.
			\item  If $N(\varepsilon_{2p})=1$, then $\{\varepsilon_{2}, \varepsilon_{p},	\sqrt{\varepsilon_{2p}}\}$ is a fundamental system of units
			of $k_1$.
		\end{enumerate}
	\end{lemma}

	Now we recall the following  useful lemmas:

	% 	\begin{lemma}[\cite{ZAT-15}]\label{3:105}
	% 		Let $d\equiv1\pmod4$ be a positive square free integer and   $\varepsilon_d=x+y\sqrt d$ be the fundamental unit of  $\QQ(\sqrt d)$. Assume   $N(\varepsilon_d)=1$,  then
	%		\begin{enumerate}[\rm\indent1.]
	%			\item $x+1$ and $x-1$ are not squares in  $\NN$,  i.e.,  $2\varepsilon_{d}$ is not a square in  $\QQ(\sqrt{d})$.
	%			\item For all prime  $p$ dividing   $d$,  $p(x+1)$ and $p(x-1)$ are not squares in  $\NN$.
	%		\end{enumerate}
	%	\end{lemma}

	% Let us close this section by recalling the   following class number formula for a multiquadratic number field  which is usually attributed to Kuroda \cite{Ku-50}, but it goes back to Herglotz \cite{He-22}  (cf. \cite[p. 27]{BFL}).
	\begin{lemma}[\cite{Ku-50}]\label{wada's f.}
		Let $K$ be a multiquadratic number field of degree $2^n$, $n\in\mathds{N}$,  and $k_i$ the $s=2^n-1$ quadratic subfields of $K$. Then
		$$h(K)=\frac{1}{2^v}( E_K: \prod_{i=1}^{s}E_{k_i})\prod_{i=1}^{s}h(k_i),$$
		with $$v=\left\{ \begin{array}{cl}
			n(2^{n-1}-1); &\text{ if } K \text{ is real, }\\
			(n-1)(2^{n-2}-1)+2^{n-1}-1 & \text{ if } K \text{ is imaginary.}
		\end{array}\right.$$
	\end{lemma}

	\begin{lemma}\label{class numbers of quadratic field}
		Let $q\equiv 3\pmod 4$ and  $p\equiv 1\pmod 4$ be two distinct  primes. 
		\begin{enumerate}[\rm 1)]
			\item By \cite[Corollary 18.4]{connor88}, we have  $h_2(p)=h_2(q)=h_2(2q)= h_2(2)=h_2(-2)=h_2(-q)=h_2(-1)=1$.
			\item If $\genfrac(){}{0}{p}{q} =-1$,   then   $h_2(pq)=h_2(2pq)=h_2(-pq)=2$, else $h_2(pq)$, $h_2(2pq)$ and $h_2(-pq)$ are divisible by $4$ $($cf. \cite[Corollaries 19.6 and 19.7]{connor88}$)$.
			\item If $ q\equiv 3\pmod 8$, then $h_2(-2q)=2$ $($cf. \cite[Corollary 19.6]{connor88}$)$.
			%	\item If $p\equiv  -q\equiv 5\pmod 8$, then $h_2(-p)=h_2(-2p)=h_2(-2q)=2$.
		\end{enumerate}
	\end{lemma}

	\section{\bf Unit groups computation}
	%In  this section, we shall compute the unit group of the fields $ \QQ(\sqrt 2, \sqrt{p}, \sqrt{q} )$ using case by case.

 	 	\subsection{\bf The case: $p\equiv1\pmod 8$, $ q \equiv7\pmod8$  and 
 	 	$\genfrac(){}{0}{p}{q} =-1$.}\text{\;}\\
 	 The following lemmas   are very useful in what follows to prove our first main theorem.
 	 
 	\begin{lemma}\label{lm expressions of units under cond 2}
		Let        $p\equiv1\pmod 8$ and $ q \equiv7\pmod8$ be two primes such that
		$\genfrac(){}{0}{p}{q} =-1$.
		\begin{enumerate}[\rm 1)]
			\item Let  $x$ and $y$   be two integers such that
			$ \varepsilon_{2pq}=x+y\sqrt{2pq}$. Then  
			\begin{enumerate}[\rm i.]
				\item $(x+1)$ is a square in $\NN$, 
				\item $\sqrt{2\varepsilon_{2pq}}=y_1+y_2\sqrt{2pq}$ and 	$2= y_1^2-2pqy_2^2$, for some integers $y_1$ and $y_2$.
			\end{enumerate}
			\item  Let    $v$ and $w$ be two integers such that
			$ \varepsilon_{pq}=v+w\sqrt{pq}$. Then   
			\begin{enumerate}[\rm i.]
				\item $(v+1)$ is a square in $\NN$, 
				\item   $\sqrt{ 2\varepsilon_{ pq}}=w_1+w_2\sqrt{pq}$ and 	$2=w_1^2-pqw_2^2$, for some integers $w_1$ and $w_2$.
			\end{enumerate}
		\end{enumerate}
	\end{lemma}
 \begin{proof}
 		As it is known that $N(\varepsilon_{2pq})=1$, then, by the unique factorization in $\mathbb{Z}$ and Lemma \ref{1:046}  there exist some integers $y_1$ and $y_2$  $(y=y_1y_2)$ such that
 	$$(1):\ \left\{ \begin{array}{ll}
 		x\pm1=y_1^2\\
 		x\mp1=2pqy_2^2,
 	\end{array}\right. \quad
 	(2):\ \left\{ \begin{array}{ll}
 		x\pm1=py_1^2\\
 		x\mp1=2qy_2^2,
 	\end{array}\right.\quad
 	\text{ or }\quad
 	(3):\ \left\{ \begin{array}{ll}
 		x\pm1=2py_1^2\\
 		x\mp1=qy_2^2,
 	\end{array}\right.
 	$$		
 	\begin{enumerate}[\rm$*$]
 		\item System $(2)$ can not occur since it implies  
 		$-1=\left(\frac{2qy_1^2}{p}\right)=\left(\frac{x\mp 1}{p}\right)=\left(\frac{x\pm1\mp 2}{p}\right)=\left(\frac{\mp2}{p}\right)=\left(\frac{2}{p}\right)=1$, which is absurd. 
 		\item We similarly show that  System $(3)$ 
 		can not occur. 
 		
 		\item Assume that $\left\{ \begin{array}{ll}
 		x-1=y_1^2\\
 		x+1=2pqy_2^2
 	\end{array}\right. .$ So 	$1=\left(\frac{ y_1^2}{q}\right)=\left(\frac{x- 1}{q}\right)=\left(\frac{x+1- 2}{q}\right)=\left(\frac{-2}{q}\right) =-1$, which is a contradiction.	
 		
 	\end{enumerate}

  	Therefore $\left\{ \begin{array}{ll}
 		x+1=y_1^2\\
 		x-1=2pqy_2^2
 	\end{array}\right.$ which gives the first item. The proof of the second item is analogous.
 \end{proof}

	\begin{lemma}\label{lm expressions of units q_2q}
		Let         $ q \equiv7\pmod8$ be a prime number.
		\begin{enumerate}[\rm 1)]
			\item  Let    $c $ and $d$ be two integers such that
			$ \varepsilon_{2q}=c +d\sqrt{2q}$. Then    
			\begin{enumerate}[\rm i.]
				\item   $c+1$ is a square in $\NN$,
				\item  $\sqrt{ 2\varepsilon_{  2q}}=d_1 +d_2\sqrt{2q}$ and 	$2=d_1^2-2qd_2^2$, for some integers $d_1$ and $d_2$.
			\end{enumerate}
			\item  Let    $\alpha $ and $\beta$ be two integers such that
			$ \varepsilon_{q}=\alpha +\beta\sqrt{q}$. Then    
			\begin{enumerate}[\rm i.]
				\item   $\alpha+1$ is a square in $\NN$,
				\item  $\sqrt{ 2\varepsilon_{  q}}=\beta_1 +\beta_2\sqrt{q}$ and 	$2=\beta_1^2-q\beta_2^2$, for some integers $\beta_1$ and $\beta_2$.
			\end{enumerate}
		\end{enumerate}
		Furthermore, for any prime number $p\equiv 1\pmod 4$ we have: 
		\begin{eqnarray}\label{norm q=7 p=1mod4q} 
			\begin{tabular}{ |c|c|c|c|c|c|c|c|c|}
				\hline
				$\varepsilon$&$\varepsilon_{2}$ &$ {\varepsilon_{p}}$&$\sqrt{\varepsilon_{q}}$&$\sqrt{\varepsilon_{2q}}$          \\
				\hline 
				$\varepsilon^{1+\tau_1}$&$-1$ &$\varepsilon_{p}^2$&	$-\varepsilon_{q}$&$-1$\\
				\hline
				$\varepsilon^{1+\tau_2}$&$\varepsilon_{2}^2$ &$-1$&	$\varepsilon_{q}$&$\varepsilon_{2q}$\\
				\hline
				$\varepsilon^{1+\tau_3}$&$\varepsilon_{2}^2$ &$\varepsilon_{p}^2$& $1$&$1$\\
				\hline
				$\varepsilon^{1+\tau_1\tau_2}$&$-1$ &$-1$&	$-\varepsilon_{q}$&$-1$\\
				\hline
				$\varepsilon^{1+\tau_1\tau_3}$&$-1$ &$\varepsilon_{p}^2$&	$-1$&$-\varepsilon_{2q}$ \\
				\hline
				$\varepsilon^{1+\tau_2\tau_3}$&$\varepsilon_{2}^2$ &$-1$&	$1$&$1$\\
				\hline
			\end{tabular} 
		\end{eqnarray}
	\end{lemma}
	\begin{proof}
		For   the two items see \cite[Lemma 4.1]{chemszekhniniaziziUnits1}. The computations in the table follows from the definitions of $\tau_i$ and the two items.
	\end{proof}

	\begin{theorem}\label{T_7_-1} Let $p\equiv 1\pmod{8}$ and $q\equiv7\pmod 8$ be two primes such that     $\genfrac(){}{0}{p}{q} =-1$. 
		Put     $\KK=\QQ(\sqrt 2, \sqrt{p}, \sqrt{q} )$.  
		\begin{enumerate}[\rm 1)]
			\item If $N(\varepsilon_{2p})=-1$,  then
			\begin{enumerate}[\rm $\bullet$]
				
				\item The unit group of $\KK$ is :
				$$E_{\KK}=\langle -1,  \varepsilon_{2}, \varepsilon_{p},   \sqrt{\varepsilon_{q}}, \sqrt{\varepsilon_{2q}},  \sqrt{\varepsilon_{pq}} , \sqrt{\varepsilon_{2}\varepsilon_{p}\varepsilon_{2p}}, 
				\sqrt{\sqrt{\varepsilon_{q}} \sqrt{\varepsilon_{2q}} \sqrt{\varepsilon_{pq}} \sqrt{\varepsilon_{2pq}}}     \rangle.$$
				\item The $2$-class group of  $\KK$ is cyclic of order  $\frac 12h_2(2p)$. 
			\end{enumerate}
			\item If $N(\varepsilon_{2p})=1$,   then
			
			\begin{enumerate}[\rm $\bullet$]
				\item The unit group of $\KK$ is :
				$$E_{\KK}=\langle -1,     \varepsilon_{2}, \varepsilon_{p},   \sqrt{\varepsilon_{q}}, \sqrt{\varepsilon_{2q}},  \sqrt{\varepsilon_{pq}} , \sqrt{\varepsilon_{2}^a\varepsilon_{p}^a \sqrt{\varepsilon_{q}} \sqrt{\varepsilon_{pq}}
					\sqrt{\varepsilon_{2p}}},
				\sqrt{\varepsilon_{2}^a\varepsilon_{p}^a  \sqrt{\varepsilon_{2q}} \sqrt{\varepsilon_{2pq}}
					\sqrt{\varepsilon_{2p}}}   \rangle,$$
				where $a\in\{0,1\}$ such that $a\equiv u+1\pmod 2$ and $u$ is defined in Lemma \ref{lm noms esp_2p}.
				\item  The $2$-class group of  $\KK$ is cyclic of order  $h_2(2p)$.
			\end{enumerate}
		\end{enumerate}
	\end{theorem}
	\begin{proof}
		
		\begin{enumerate}[\rm 1)]
			\item  %	To prove this theorem, we use the methods given by \cite{wada, azizunints99} . %we use the algorithm described by \cite{wada}.
			Assume that   $N(\varepsilon_{2p})=-1$.  By Lemma \ref{units of k1},     $\{\varepsilon_{2}, \varepsilon_{p},	\sqrt{\varepsilon_{2}\varepsilon_{p}\varepsilon_{2p}}\}$  is a  fundamental system of units of $k_1$. Using Lemmas \ref{lm expressions of units q_2q} and \ref{lm expressions of units under cond 2}, we check   that $\{\varepsilon_{2}, \sqrt{\varepsilon_{q}}, \sqrt{\varepsilon_{2q}}\}$ and $\{ \varepsilon_{2}, 	\sqrt{\varepsilon_{pq}},\sqrt{\varepsilon_{2pq}}\}$ are respectively fundamental systems of units of $k_2$ and $k_3$.
			It follows that,  	$$E_{k_1}E_{k_2}E_{k_3}=\langle-1,  \varepsilon_{2}, \varepsilon_{p},   \sqrt{\varepsilon_{q}}, \sqrt{\varepsilon_{2q}},  \sqrt{\varepsilon_{pq}} ,\sqrt{ \varepsilon_{2pq}}, \sqrt{\varepsilon_{2}\varepsilon_{p}\varepsilon_{2p}}\rangle.$$	
			Thus we shall determine elements of $E_{k_1}E_{k_2}E_{k_3}$ which are squares in $\KK$. Let  $\xi$ be an element of $\KK$ which is a  square root of an element of $E_{k_1}E_{k_2}E_{k_3}$. We can assume that
			$$\xi^2=\varepsilon_{2}^a\varepsilon_{p}^b \sqrt{\varepsilon_{q}}^c\sqrt{\varepsilon_{2q}}^d\sqrt{\varepsilon_{pq}}^e\sqrt{\varepsilon_{2pq}}^f
			\sqrt{\varepsilon_{2}\varepsilon_{p}\varepsilon_{2p}}^g,$$
			where $a, b, c, d, e, f$ and $g$ are in $\{0, 1\}$.

			\noindent\ding{224}  Let us start	by applying   the norm map $N_{\KK/k_2}=1+\tau_2$. We have 
			$\sqrt{\varepsilon_{pq}}^{1+ \tau_2}=1$,
			$\sqrt{\varepsilon_{2pq}}^{1+ \tau_2}=1$ and $\sqrt{\varepsilon_{2}\varepsilon_{p}\varepsilon_{2p}}^{1+ \tau_2}=(-1)^v\varepsilon_{2}$, for some $v\in\{0,1\}$. Thus, by the table \eqref{norm q=7 p=1mod4q} we have:
			%	\begin{eqnarray}\label{T_7_-1_tau2_N=-1}
			%	\begin{tabular}{ |c|c|c|c|c|c|c|c|c|}
			%		\hline
			%		$\varepsilon$&$\varepsilon_{2}$ &$ {\varepsilon_{p}}$&$\sqrt{\varepsilon_{q}}$&$\sqrt{\varepsilon_{2q}}$&$\sqrt{\varepsilon_{pq}}$&$\sqrt{\varepsilon_{2pq}}$&$\sqrt{\varepsilon_{2}\varepsilon_{p}\varepsilon_{2p}}$ \\
			%		\hline
			%		$\varepsilon^{1+\tau_2}$&$\varepsilon_{2}^2$ &$-1$&	$\varepsilon_{q}$&$\varepsilon_{2q}$&$ 1$& $1$&$(-1)^v\varepsilon_{2}$ \\
			%		\hline
			%	\end{tabular} 
			%\end{eqnarray}
			\begin{eqnarray*}
				N_{\KK/k_2}(\xi^2)&=&
				\varepsilon_{2}^{2a}\cdot(-1)^b \cdot \varepsilon_{q}^c\cdot\varepsilon_{2q}^d\cdot1\cdot1 \cdot (-1)^{gv}\varepsilon_{2}^g\\
				&=&	\varepsilon_{2}^{2a}  \varepsilon_{q}^c\varepsilon_{2q}^d\cdot(-1)^{b+gv} \varepsilon_{2}^g.
			\end{eqnarray*}
			
			Thus    $b+gv\equiv 0\pmod2$ and $g=0$. Therefore,  $b=0$ and
			$$\xi^2=\varepsilon_{2}^a  \sqrt{\varepsilon_{q}}^c\sqrt{\varepsilon_{2q}}^d\sqrt{\varepsilon_{pq}}^e\sqrt{\varepsilon_{2pq}}^f.$$

			\noindent\ding{224} Let us apply the norm $N_{\KK/k_5}=1+\tau_1\tau_2$, with $k_5=\QQ(\sqrt{q}, \sqrt{2p})$. We have $\sqrt{\varepsilon_{pq}}^{1+ \tau_1\tau_2}=-1$ 
			and $\sqrt{\varepsilon_{2pq}}^{1+\tau_1\tau_2}={-\varepsilon_{2pq}}$. Thus, by the table \eqref{norm q=7 p=1mod4q} we have:
			%	\begin{eqnarray}\label{T_7_-1_tau1tau2_N=-1}
			%	\begin{tabular}{ |c|c|c|c|c|c|c|c|c|}
			%		\hline
			%		$\varepsilon$&$\varepsilon_{2}$ &$ {\varepsilon_{p}}$&$\sqrt{\varepsilon_{q}}$&$\sqrt{\varepsilon_{2q}}$&$\sqrt{\varepsilon_{pq}}$&$\sqrt{\varepsilon_{2pq}}$ \\
			%		\hline
			%		$\varepsilon^{1+\tau_1\tau_2}$&$-1$ &$-1$&	$-\varepsilon_{q}$&$-1$&$ -1$& ${-\varepsilon_{2pq}}$   \\
			%		\hline
			%	\end{tabular}
			%	\end{eqnarray}	
			\begin{eqnarray*}
				N_{\KK/k_5}(\xi^2)&=&(-1)^a\cdot (-1)^c\cdot \varepsilon_{  q}^c\cdot(-1)^d\cdot(-1)^e\cdot (-1)^f\cdot \varepsilon_{  2pq}^f\\
				&=&	 (-1)^{a+c+d+e+f}\varepsilon_{  q}^c\cdot \varepsilon_{  2pq}^f.
			\end{eqnarray*}
			Thus $a+c+d+e+f=0\pmod 2$ and $f=c$. Thus, $a+d+e=0\pmod 2$. Therefore, 
			$$\xi^2=\varepsilon_{2}^a  \sqrt{\varepsilon_{q}}^c\sqrt{\varepsilon_{2q}}^d\sqrt{\varepsilon_{pq}}^e\sqrt{\varepsilon_{2pq}}^c.$$

			\noindent\ding{224} Let us apply the norm $N_{\KK/k_6}=1+\tau_1\tau_3$, with $k_6=\QQ(\sqrt{p}, \sqrt{2q})$. We have    $\sqrt{\varepsilon_{pq}}^{1+ \tau_1\tau_3}=-1$ 
			and $\sqrt{\varepsilon_{2pq}}^{1+\tau_1\tau_3}={-\varepsilon_{2pq}}$. Thus, by the table \eqref{norm q=7 p=1mod4q} we have:
			%	\begin{eqnarray}\label{T_7_-1_tau1tau3_N=-1}
			%		\begin{tabular}{ |c|c|c|c|c|c|c|c|c|}
			%			\hline
			%			$\varepsilon$&$\varepsilon_{2}$ &$ {\varepsilon_{p}}$&$\sqrt{\varepsilon_{q}}$&$\sqrt{\varepsilon_{2q}}$&$\sqrt{\varepsilon_{pq}}$&$\sqrt{\varepsilon_{2pq}}$ \\
			%			\hline
			%			$\varepsilon^{1+\tau_1\tau_3}$&$-1$ &$\varepsilon_{p}^2$&	$-1$&$-\varepsilon_{2q}$&$ -1$& ${-\varepsilon_{2pq}}$   \\
			%			\hline
			%		\end{tabular} 
			%	\end{eqnarray}
			\begin{eqnarray*}
				N_{\KK/k_5}(\xi^2)&=&(-1)^a\cdot (-1)^c\cdot (-1)^d\cdot \varepsilon_{2  q}^d\cdot (-1)^e\cdot(-1)^c\cdot \varepsilon_{  2pq}^c\\
				&=&	 (-1)^{a+d+e}\varepsilon_{2  q}^d\varepsilon_{  2pq}^c.
			\end{eqnarray*}
			Thus $a+d+e=0\pmod 2$ and 
			$d=c$. Therefore   
			$$\xi^2=\varepsilon_{2}^a  \sqrt{\varepsilon_{q}}^c\sqrt{\varepsilon_{2q}}^c\sqrt{\varepsilon_{pq}}^e\sqrt{\varepsilon_{2pq}}^c.$$
			
			\noindent\ding{224} Let us apply the norm $N_{\KK/k_3}=1+\tau_2\tau_3$, with $k_3=\QQ(\sqrt{2}, \sqrt{pq})$.  We have    $\sqrt{\varepsilon_{pq}}^{1+ \tau_2\tau_3}={\varepsilon_{pq}}$ 
			and $\sqrt{\varepsilon_{2pq}}^{1+\tau_2\tau_3}={ \varepsilon_{2pq}}$. Thus, by the table  \eqref{norm q=7 p=1mod4q} we have: 
			%	\begin{eqnarray}\label{T_7_-1_tau2tau3_N=-1}
			%	\begin{tabular}{ |c|c|c|c|c|c|c|c|c|}
			%		\hline
			%		$\varepsilon$&$\varepsilon_{2}$ &${\varepsilon_{p}}$&$\sqrt{\varepsilon_{q}}$&$\sqrt{\varepsilon_{2q}}$&$\sqrt{\varepsilon_{pq}}$&$\sqrt{\varepsilon_{2pq}}$ \\
			%		\hline
			%		$\varepsilon^{1+\tau_2\tau_3}$&$\varepsilon_{2}^2$ &$-1$&	$1$&$1$&$ {\varepsilon_{pq}}$& ${\varepsilon_{2pq}}$   \\
			%		\hline
			%	\end{tabular}
			%	\end{eqnarray}
			\begin{eqnarray*}
				N_{\KK/k_3}(\xi^2)&=&\varepsilon_{2}^{2a}\cdot   1\cdot 1\cdot \varepsilon_{pq}^e \cdot \varepsilon_{  2pq}^c\\
				&=& \varepsilon_{2}^{2a} \varepsilon_{pq}^c   \varepsilon_{  2pq}^c.
			\end{eqnarray*}
			We have nothing to deduce from this. Therefore, we apply another norm.

			\noindent\ding{224} Let us apply the norm $N_{\KK/k_4}=1+\tau_1$, with $k_4=\QQ(\sqrt{p}, \sqrt{q})$.  We have    $\sqrt{\varepsilon_{pq}}^{1+ \tau_1}={-\varepsilon_{pq}}$ 
			and $\sqrt{\varepsilon_{2pq}}^{1+\tau_1}=-1$. Thus, by the table \eqref{norm q=7 p=1mod4q} we have: 
			%	\begin{eqnarray}\label{T_7_-1_tau1_N=-1}
			%	\begin{tabular}{ |c|c|c|c|c|c|c|c|c|}
			%		\hline
			%			$\varepsilon$&$\varepsilon_{2}$ &${\varepsilon_{p}}$&$\sqrt{\varepsilon_{q}}$&$\sqrt{\varepsilon_{2q}}$&$\sqrt{\varepsilon_{pq}}$&$\sqrt{\varepsilon_{2pq}}$ \\
			%			\hline
			%		$\varepsilon^{1+\tau_1}$&$-1$ &$\varepsilon_{p}^2$&	$-\varepsilon_{q}$&$-1$&$ {-\varepsilon_{pq}}$& $-1$   \\
			%			\hline
			%		\end{tabular}
			%	\end{eqnarray}
			%So we have 		
			\begin{eqnarray*}
				N_{\KK/k_5}(\xi^2)&=&(-1)^a\cdot (-1)^c\cdot \varepsilon_{  q}^c\cdot (-1)^c\cdot (-1)^e\cdot \varepsilon_{   pq}^e\cdot(-1)^c\cdot \\
				&=&	 (-1)^{a+c+e} \varepsilon_{  q}^c\varepsilon_{   pq}^e.
			\end{eqnarray*}
			Thus $a+c+e=0\pmod 2$ and $c=e$. Hence, $a=0$ and
			$$\xi^2=   \sqrt{\varepsilon_{q}}^c\sqrt{\varepsilon_{2q}}^c\sqrt{\varepsilon_{pq}}^c\sqrt{\varepsilon_{2pq}}^c.$$
		
		%%%%%%%%%%%%%%%%%%%%%%%%%%%%%%%%%%%%%%%%%%%%%%%%%%%%%%%%%%%%%%%%%%%%%%%%%
		
		Let us show that  the square root of $\sqrt{\varepsilon_{q}}\sqrt{\varepsilon_{2q}}\sqrt{\varepsilon_{pq}}\sqrt{\varepsilon_{2pq}}$ is an element of $\KK$.	
		Note that one can easily check that the  $2$-class group of $k_5=\mathbb Q(\sqrt{2p}, \sqrt{q})$ is cyclic and by   Lemmas  \ref{wada's f.} and \ref{class numbers of quadratic field}, we have   $h_2(k_5)=\frac{1}{4}q(k_5)h_2(2p)h_2(q)h_2(2pq)=\frac{1}{2}q(k_5)h_2(2p) $.
		Using Lemmas \ref{lm expressions of units under cond 2} and \ref{lm expressions of units q_2q} (and the algorithm given in page  \pageref{algo wada}), we easily deduce that  that $q(k_5)=2$. Thus $h_2(k_5)=h_2(2p)$. Since $\KK/k_5$ is an unramified quadratic extension, then
		
		\begin{eqnarray}\label{class nbr of K+}
			h_2(\KK)=\frac{1}{2}\cdot h_2(k_5)=\frac{1}{2}\cdot h_2(2p) .
		\end{eqnarray}
		Assume by absurd that $\sqrt{\varepsilon_{q}}\sqrt{\varepsilon_{2q}}\sqrt{\varepsilon_{pq}}\sqrt{\varepsilon_{2pq}}$ is not a square in $\KK$. Then $q(\KK)=2^5$.	
		By Lemma  \ref{wada's f.}, we have:
		\begin{eqnarray}
			h_2(\KK)&=&\frac{1}{2^{9}}q(\KK)  h_2(2) h_2(p) h_2(q)h_2(2p) h_2(2q)h_2(pq)  h_2(2pq)  \\
			&=&\frac{1}{2^{9}}\cdot 2^5\cdot 1\cdot 1 \cdot 1  \cdot h_2(2p) \cdot 1 \cdot 2 \cdot 2=\frac{1}{4}\cdot h_2(2p).\nonumber
		\end{eqnarray}
		Which is a contradiction with	\eqref{class nbr of K+}. Therefore $c=1$ and $\sqrt{\varepsilon_{q}} \sqrt{\varepsilon_{2q}} \sqrt{\varepsilon_{pq}} \sqrt{\varepsilon_{2pq}}$\label{is a square} is a square in $\KK$. So the first item.

		%%%%%%%%%%%%%%%%%%%%%%%%%%%%%%%%%%%%%%%%%%%
		%	As in the proof of Theorem \ref{T_3_-1} (cf. page \pageref{is a square}), we show that 
		%	$\sqrt{\varepsilon_{q}} \sqrt{\varepsilon_{2q}} \sqrt{\varepsilon_{pq}} \sqrt{\varepsilon_{2pq}}$ is a square in $\KK$ and that the $2$-class group of $\KK$ is cyclic of order $\frac 12 h_2(2p)$.
			
			\item    Assume that	$N(\varepsilon_{2p})=1$.  Then by Lemma \ref{units of k1},     $\{\varepsilon_{2}, \varepsilon_{p},	\sqrt{\varepsilon_{2p}}\}$  is a  fundamental system of units of $k_1$ and   from Lemmas  \ref{lm expressions of units under cond 2} and \ref{lm expressions of units q_2q} we deduce that $\{\varepsilon_{2}, \sqrt{\varepsilon_{q}}, \sqrt{\varepsilon_{2q}}\}$ and $\{ \varepsilon_{2}, 	\sqrt{\varepsilon_{pq}},\sqrt{\varepsilon_{2pq}}\}$ are respectively fundamental systems of units of $k_2$ and $k_3$.
			So we have:  	$$E_{k_1}E_{k_2}E_{k_3}=\langle-1,  \varepsilon_{2}, \varepsilon_{p},   \sqrt{\varepsilon_{q}}, \sqrt{\varepsilon_{2q}},  \sqrt{\varepsilon_{pq}} ,\sqrt{ \varepsilon_{2pq}}, \sqrt{\varepsilon_{2p}}\rangle.$$	
			Put
			$$\xi^2=\varepsilon_{2}^a\varepsilon_{p}^b \sqrt{\varepsilon_{q}}^c\sqrt{\varepsilon_{2q}}^d\sqrt{\varepsilon_{pq}}^e\sqrt{\varepsilon_{2pq}}^f
			\sqrt{\varepsilon_{2p}}^g,$$
			where $a, b, c, d, e, f$ and $g$ are in $\{0, 1\}$. We shall proceed as in the first item. Assume that $\xi\in \KK$.
			
			\noindent\ding{224}	Let us start	by applying   the norm map $N_{\KK/k_2}=1+\tau_2$. We have	
			\begin{eqnarray*}
				N_{\KK/k_2}(\xi^2)&=&
				\varepsilon_{2}^{2a}\cdot (-1)^b \cdot \varepsilon_{q}^c\cdot \varepsilon_{2q}^d\cdot1\cdot1 \cdot (-1)^{gu}\\\
				&=&	\varepsilon_{2}^{2a}  \varepsilon_{q}^c\varepsilon_{2q}^d\cdot(-1)^{b +gu} .
			\end{eqnarray*}
			Thus,  	  $	b +gu\equiv 0\pmod2$.    
			
			\noindent\ding{224}	Let   us apply    the norm map $N_{\KK/k_5}=1+\tau_1\tau_2$, with $k_5=\QQ(\sqrt{q}, \sqrt{2p})$.	We have	
			\begin{eqnarray*}
				N_{\KK/k_5}(\xi^2)&=&
				(-1)^a  \cdot (-1)^b \cdot (-1)^c  \cdot\varepsilon_{q}^c\cdot(-1)^d\cdot(-1)^e \cdot(-1)^f \cdot\varepsilon_{2pq}^f\cdot(-1)^g\cdot \varepsilon_{2p}^{g} \\\
				&=&\varepsilon_{2p}^{g} (-1)^{a+b+c+d+e+f+g } \cdot\varepsilon_{q}^c\varepsilon_{2pq}^f\varepsilon_{2p}^{g}.
			\end{eqnarray*}
			Thus, $a+b+c+d+e+f+g\equiv 0\pmod2$ and  $c+f+ g\equiv 0\pmod2$.	Therefore, $a+b +d+e  \equiv 0\pmod2$.
			
			\noindent\ding{224}	Let    us apply      the norm map $N_{\KK/k_6}=1+\tau_1\tau_3$, with $k_6=\QQ(\sqrt{p}, \sqrt{2q})$.	We have	
			\begin{eqnarray*}
				N_{\KK/k_6}(\xi^2)&=&
				(-1)^a  \cdot \varepsilon_{p}^{2b} \cdot (-1)^c \cdot(-1)^d \cdot\varepsilon_{2q}^d\cdot (-1)^e  \cdot(-1)^f \cdot\varepsilon_{2pq}^f\cdot (-1)^{gu+g}\\\
				&=&\varepsilon_{p}^{2b}\cdot (-1)^{a +c+ d+e+f+ug+g} \cdot\varepsilon_{2q}^d \varepsilon_{2pq}^f.
			\end{eqnarray*}
			Thus, $a +c+ d+e+f+ug+g\equiv 0\pmod2$ and $d=f$. Then,   $a +c +e +ug+g\equiv 0\pmod2$ and 
			$$\xi^2=\varepsilon_{2}^a\varepsilon_{p}^b \sqrt{\varepsilon_{q}}^c\sqrt{\varepsilon_{2q}}^d\sqrt{\varepsilon_{pq}}^e\sqrt{\varepsilon_{2pq}}^d
			\sqrt{\varepsilon_{2p}}^g,$$
			
			\noindent\ding{224} By applying      the norm map $N_{\KK/k_3}=1+\tau_2\tau_3$, with $k_3=\QQ(\sqrt{2}, \sqrt{pq})$, we deduce nothing new.		
			
			\noindent\ding{224} Let us apply the norm $N_{\KK/k_4}=1+\tau_1$, with $k_4=\QQ(\sqrt{p}, \sqrt{q})$. We have 
			\begin{eqnarray*}
				N_{\KK/k_5}(\xi^2)&=&(-1)^a\cdot \varepsilon_{p}^{2b} \cdot (-1)^c\cdot \varepsilon_{  q}^c\cdot (-1)^d\cdot (-1)^e\cdot \varepsilon_{   pq}^e\cdot(-1)^d\cdot  (-1)^{gu+g}\\
				&=&	\varepsilon_{p}^{2b} (-1)^{a+c+ e +gu+g} \varepsilon_{  q}^c\varepsilon_{   pq}^e.
			\end{eqnarray*}
			Thus, $a+c +e+ gu+g\equiv 0\pmod2$ and $c=e$. Then, $a   +gu+g\equiv 0\pmod2$ and
			$$\xi^2=\varepsilon_{2}^a\varepsilon_{p}^b \sqrt{\varepsilon_{q}}^e\sqrt{\varepsilon_{2q}}^d\sqrt{\varepsilon_{pq}}^e\sqrt{\varepsilon_{2pq}}^d
			\sqrt{\varepsilon_{2p}}^g,$$
			with
			\begin{eqnarray}
				b +gu\equiv 0\pmod2 \label{2eq1}\\
				e+d+ g\equiv 0\pmod2\label{2eqs}\\
				a+b +d+e  \equiv 0\pmod2 \label{2eq2}\\
				a   +ug+g\equiv 0\pmod2\label{2eq3}
			\end{eqnarray}
			On the other hand, as in the proof of the first item, we show that $h_2(\KK)=h_2(2p)$ and that $q(\KK)=2^7 $. So if $g=0$, then by equalities \eqref{2eq1} and \eqref{2eq3} $a=b=0$ and   by equality 
			\eqref{2eq2}, we get $d=e$. Thus,  $\xi^2=  \sqrt{\varepsilon_{q}}^e\sqrt{\varepsilon_{2q}}^e\sqrt{\varepsilon_{pq}}^e\sqrt{\varepsilon_{2pq}}^e$, with $e=0$ or $1$. In the two cases we have $q(\KK)\not=2^7 $. Therefore, $g=1$ and so by \eqref{2eqs} $e\not= d$. By  \eqref{2eq1}  and \eqref{2eq3}, $b=u\not=a$ and $a\equiv u+1\pmod 2$. Hence, necessarily the two  following equations have solution in $\KK$:
			$\xi^2=\varepsilon_{2}^a\varepsilon_{p}^u \sqrt{\varepsilon_{2q}} \sqrt{\varepsilon_{2pq}}\sqrt{\varepsilon_{2p}} $ and
			$\xi^2=\varepsilon_{2}^a\varepsilon_{p}^u \sqrt{\varepsilon_{q}} \sqrt{\varepsilon_{pq}}    \sqrt{\varepsilon_{2p}} $, where $a\equiv u+1\pmod 2$ and $u$ is defined in Page \pageref{the int u}. Since, $q(\KK)=2^7 $, these two equations are necessarily  solvable in $\KK$. Which completes the proof. 
		\end{enumerate}
		
	\end{proof}

	\begin{exam}Using  PARI/GP calculator version  2.15.5 (64bit), Feb 11 2024, we  get the following examples which illustrate the above theorem.
		\begin{enumerate}[\rm $1)$]
			 	\item  Consider the prime numbers $p=41$ and  $q=7$. In this case, the   conditions of the first item of the above theorem are verified and we have:
			$$E_{\KK}=\langle -1,  \varepsilon_{2}, \varepsilon_{p},   \sqrt{\varepsilon_{q}}, \sqrt{\varepsilon_{2q}},  \sqrt{\varepsilon_{pq}} , \sqrt{\varepsilon_{2}\varepsilon_{p}\varepsilon_{2p}}, 
			\sqrt{\sqrt{\varepsilon_{q}} \sqrt{\varepsilon_{2q}} \sqrt{\varepsilon_{pq}} \sqrt{\varepsilon_{2pq}}}     \rangle.$$
			Furthermore, we have $h_2(2p)=4$  and the $2$-class group of $\KK$ is isomorphic to $\ZZ/2\ZZ$.
			\item Consider the prime numbers $p=17$ and  $q=7$. In this case, the   conditions of the second item of the above theorem are verified and we have:
				$$E_{\KK}=\langle -1,     \varepsilon_{2}, \varepsilon_{p},   \sqrt{\varepsilon_{q}}, \sqrt{\varepsilon_{2q}},  \sqrt{\varepsilon_{pq}} , \sqrt{\varepsilon_{2}^a\varepsilon_{p}^a \sqrt{\varepsilon_{q}} \sqrt{\varepsilon_{pq}}
				\sqrt{\varepsilon_{2p}}},
			\sqrt{\varepsilon_{2}^a\varepsilon_{p}^a  \sqrt{\varepsilon_{2q}} \sqrt{\varepsilon_{2pq}}
				\sqrt{\varepsilon_{2p}}}   \rangle,$$
			for some 	$a\in\{0,1\}$ such that $a \equiv 1+u\pmod2$.	Furthermore,we have $h_2(2p)=2$  and the $2$-class group of $\KK$ is isomorphic to $\ZZ/2\ZZ$.
		\end{enumerate}
	\end{exam}

	%%%%%%%%%%%%%%%%%%%%%%%%%%%%%%%%%%%%%%%%%%%%%%%%%%%%%%%%%%%%%%%%%%%%	
	
	\subsection{\bf The case: $p\equiv1\pmod 8$, $ q \equiv7\pmod8$  and 
		$\genfrac(){}{0}{p}{q} =1$.}\text{\;}\\

	\begin{lemma}\label{lm expressions of units under cond 3 case: 7=1 mod 8 and ()=1}
		Let        $p\equiv1\pmod 8$ and $ q \equiv7\pmod8$ be two primes such that
		$\genfrac(){}{0}{p}{q} =1$.
		\begin{enumerate}[\rm 1)]
			\item Let  $x$ and $y$   be two integers such that
			$ \varepsilon_{2pq}=x+y\sqrt{2pq}$. Then   
			\begin{enumerate}[\rm i.]
				\item $(x+1)$, $p(x+1)$ or $2p(x+1)$ is a square in $\NN$, 
				\item Furthermore, 
				\begin{enumerate}[\rm a)]
					\item If $(x+1)$ is a square in $\NN$, then $\sqrt{2\varepsilon_{2pq}}=y_1+y_2\sqrt{2pq}$ and 	$2=  y_1^2-2pqy_2^2$.
					\item If $p(x+1)$ is a square in $\NN$, then $\sqrt{2\varepsilon_{2pq}}=y_1\sqrt{p}+y_2\sqrt{2q}$ and $2= py_1^2-2qy_2^2$.
					\item If $2p(x+1)$ is a square in $\NN$, then $\sqrt{2\varepsilon_{2pq}}=y_1\sqrt{2p}+y_2\sqrt{q}$ and $2=  2py_1^2-qy_2^2$. 
				\end{enumerate}
				
			\end{enumerate}
			\indent Where $y_1$ and $y_2$ are two integers  such that $y=y_1y_2$.
			\item  Let    $v$ and $w$ be two integers such that
			$ \varepsilon_{pq}=v+w\sqrt{pq}$. Then  we have 
			\begin{enumerate}[\rm i.]
				\item $(v+1)$, $p(v+1)$ or $2p(v+1)$ is a square in $\NN$, 
				\item  Furthermore, 
				\begin{enumerate}[\rm a)]
					\item If $(v+1)$ is a square in $\NN$, then $\sqrt{2\varepsilon_{pq}}=w_1+w_2\sqrt{pq}$ and 	$2=  w_1^2-pqw_2^2$.
					\item If $p(v+1)$ is a square in $\NN$, then $\sqrt{2\varepsilon_{pq}}=w_1\sqrt{p}+w_2\sqrt{q}$ and $2= pw_1^2-qw_2^2$.
					\item If $2p(v+1)$ is a square in $\NN$, then $\sqrt{\varepsilon_{pq}}=w_1\sqrt{p}+w_2\sqrt{q}$ and $1= pw_1^2-qw_2^2$. 
				\end{enumerate}
				
			\end{enumerate}	
			Where $w_1$ and $w_2$ are two integers  such that $w=w_1w_2$ in $a)$ and $b)$, and  $w=2w_1w_2$ in $c)$.
		\end{enumerate}
	\end{lemma}	
	\begin{proof}
		We proceed as in the proof of Lemma	\ref{lm expressions of units under cond 2}.
	\end{proof}

	{\begin{table}[H]
			\renewcommand{\arraystretch}{2.5}
			%\setlength{\tabcolsep}{1cm}
			%\footnotesize
			%\tiny		
			%\begin{center}%\rotatebox{-90}{%We construct the following table:\\
			\begin{tabular}{|c|c|c|c|c|c|c|c|c|c}
				\hline
				$\varepsilon$  & Conditions  & $\varepsilon^{1+\tau_2}$   & $\varepsilon^{1+\tau_1\tau_2}$& $\varepsilon^{1+\tau_1\tau_3}$& $\varepsilon^{1+\tau_2\tau_3}$& $\varepsilon^{1+\tau_1}$  \\ \hline

				\multirow{3}{*}{ $\sqrt{\varepsilon_{2pq}} $}	  &$(x+1)$ is a square in $\mathbb{N}$  & $1$ & $-\varepsilon_{2pq}$& $-\varepsilon_{2pq}$ & $ \varepsilon_{2pq}$& $-1$ \\ \cline{2-7}
				
				&$p(x+1)$ is a square in $\mathbb{N}$   & $-1$ & $\varepsilon_{2pq}$&$-\varepsilon_{2pq}$ & $-\varepsilon_{2pq}$& $-1$ \\ \cline{2-7}
				
				&$2p(x+1)$ is a square  in $\mathbb{N}$   & $-1$ & $-\varepsilon_{2pq}$& $\varepsilon_{2pq}$ & $-\varepsilon_{2pq}$& $1$ \\ \hline

				\multirow{3}{*}{  	$\sqrt{\varepsilon_{pq}}$} & $(v+1)$ is a square  in $\mathbb{N}$  & $ 1$ & $-1$& $-1$ & $ \varepsilon_{pq}$&$ -\varepsilon_{pq}$ \\ \cline{2-7}
				
				& $p(v+1)$ is a square  in $\mathbb{N}$  & $-1$ & $1$& $-1$ & $-\varepsilon_{pq}$& $-\varepsilon_{pq}$ \\ \cline{2-7}
				
				& $2p(v+1)$ is a square in $\mathbb{N}$ & $-1$ & $-1$& $1$ & $-\varepsilon_{pq}$&$\varepsilon_{pq}$ \\ \hline

			\end{tabular}
			\vspace*{0.2cm}
			\caption{Norms maps of units } \label{tab3 case: 7=1 mod 8 and ()=1} 
			%}
			%\caption{  Norms   when $q_1$ and $q_2$ verify conditions  $(\ref{cond 1})$ }\label{table1} 
			%	\end{center}
	\end{table} }

		Let $p\equiv 1\pmod{8}$ and $q\equiv7\pmod 8$ be two primes such that     $\genfrac(){}{0}{p}{q} =1$. Then, by Lemmas \ref{wada's f.} and  \ref{class numbers of quadratic field},  we have:
	\begin{eqnarray}\label{classnumberofKK case: 7=1 mod 8 and ()=1}
		  h_2(\KK)=\frac{1}{2^{9}}q(\KK)\cdot h_2(2p)\cdot h_2(pq)\cdot h_2(2pq).
	\end{eqnarray}
 	The above lemma shows that we have nine cases to distinguish. The following theorems treat all these cases.

	\begin{theorem}\label{T_3_1_C1}  Let $p\equiv 1\pmod{8}$ and $q\equiv7\pmod 8$ be two primes such that     $\genfrac(){}{0}{p}{q} =1$. 
		Put     $\KK=\QQ(\sqrt 2, \sqrt{p}, \sqrt{q} )$.  Assume furthermore that $x+1$ and $v+1$ are squares in $\mathbb{N}$, where $x$ and $v$ are defined in Lemma \ref{lm expressions of units under cond 3 case: 7=1 mod 8 and ()=1}.
		\begin{enumerate}[\rm 1)]
			\item If $N(\varepsilon_{2p})=-1$,  then
			\begin{enumerate}[\rm $\bullet$]	
				\item The unit group of $\KK$ is :
				$$E_{\KK}=\langle -1,  \varepsilon_{2}, \varepsilon_{p},   \sqrt{\varepsilon_{q}}, \sqrt{\varepsilon_{2q}},  \sqrt{\varepsilon_{pq}} , \sqrt{\varepsilon_{2}\varepsilon_{p}\varepsilon_{2p}}, 
				\sqrt{\sqrt{\varepsilon_{q}}^a \sqrt{\varepsilon_{2q}}^a \sqrt{\varepsilon_{pq}}^a \sqrt{\varepsilon_{2pq}}^{1+b}}\rangle$$
				where $ a$, $b\in \{0,1\}$ such that $a\not=b$ and $a =1$ if and only if $\sqrt{\varepsilon_{q}}\sqrt{\varepsilon_{2q}}\sqrt{\varepsilon_{pq}}\sqrt{\varepsilon_{2pq}}$ is a square in $\KK$.
				\item The $2$-class number  of  $\KK$ equals $ \frac{1}{2^{4-a}} h_2(2p)h_2(pq)h_2(2pq)$.
			\end{enumerate}
			\item If $N(\varepsilon_{2p})=1$  and   $a\in\{0,1\}$ such that $a \equiv 1+u\pmod2$, then
			
			\begin{enumerate}[\rm $\bullet$]
				\item  The unit group of $\KK$ is :
				$$E_{\KK}=\langle -1,  \varepsilon_{2}, \varepsilon_{p},   \sqrt{\varepsilon_{q}}, \sqrt{\varepsilon_{2q}},  \sqrt{\varepsilon_{pq}} , \sqrt{\varepsilon_{2}^{ar'}\varepsilon_{p}^{ar'} \sqrt{\varepsilon_{q}}^{r'} \sqrt{\varepsilon_{pq}}^{r'}
					\sqrt{\varepsilon_{2p}}^{1+s'}},
				\sqrt{\varepsilon_{2}^{ar}\varepsilon_{p}^{ar}  \sqrt{\varepsilon_{2q}}^{r} \sqrt{\varepsilon_{2pq}}^{1+s}
					\sqrt{\varepsilon_{2p}}^r}      \rangle$$
				where $ r, r', s$, $s'\in \{0,1\}$ such that $r\not=s$ $($resp. $r'\not=s'$$)$ and $r =1$ $($resp.  $r' =1$$)$ if and only if  $ \varepsilon_{2}^{a}\varepsilon_{p}^{a}  \sqrt{\varepsilon_{2q}} \sqrt{\varepsilon_{2pq}}
				\sqrt{\varepsilon_{2p}}$ (resp. $\varepsilon_{2}^{a}\varepsilon_{p}^{a}  \sqrt{\varepsilon_{ q}} \sqrt{\varepsilon_{pq}}
				\sqrt{\varepsilon_{2p}}$) is a square in $\KK$.
				\item  The $2$-class number  of  $\KK$ equals $ \frac{1}{2^{4-r-r'}} h_2(2p)h_2(pq)h_2(2pq)$.
			\end{enumerate}
		\end{enumerate}
	\end{theorem}
	 \begin{proof}
	 	The same computations as in the proof of Theorem \ref{T_7_-1} give the result.
	 \end{proof}
	
	%%%%%%%%%%%%%%%%%%%%%%%%%%%%%%%%%%%%%%%%%%%%%%%%%%%%%%%%%%%%%%%%%%%%%%%%%%%%%%%%%%%%%%%%%%%%%%%%%%%%%%%%%%%%%

	\begin{exam}Using  PARI/GP calculator version  2.15.5 (64bit), Feb 11 2024, we  get the following examples which illustrate the above theorem.
			\begin{enumerate}[\rm $1)$]
	%	\item Let $p=463$ and  $q=113$. We have   $p\equiv 1\pmod{8}$ and $q\equiv7\pmod 8$,     $\genfrac(){}{0}{p}{q} =1$,  
	 %	$x+1$ and $v+1$ are squares in $\mathbb{N}$, where $x$ and $v$ are defined in Lemma \ref{lm expressions of units under cond 3 case: 7=1 mod 8 and ()=1} and   $N(\varepsilon_{2p})=-1$. In this case we have :
	 %	$$E_{\KK}=\langle -1,  \varepsilon_{2}, \varepsilon_{p},   \sqrt{\varepsilon_{q}}, \sqrt{\varepsilon_{2q}},  \sqrt{\varepsilon_{pq}} , \sqrt{\varepsilon_{2}\varepsilon_{p}\varepsilon_{2p}}, 
	 %	\sqrt{\sqrt{\varepsilon_{q}}  \sqrt{\varepsilon_{2q}}  \sqrt{\varepsilon_{pq}}  \sqrt{\varepsilon_{2pq}}}\rangle.$$
	 	\item Consider the prime numbers $p=113$ and  $q=463$. In this case, the   conditions of the first item of the above theorem are verified and we have:
		$$E_{\KK}=\langle -1,  \varepsilon_{2}, \varepsilon_{p},   \sqrt{\varepsilon_{q}}, \sqrt{\varepsilon_{2q}},  \sqrt{\varepsilon_{pq}} , \sqrt{\varepsilon_{2}\varepsilon_{p}\varepsilon_{2p}}, 
		\sqrt{\sqrt{\varepsilon_{q}}  \sqrt{\varepsilon_{2q}}  \sqrt{\varepsilon_{pq}}  \sqrt{\varepsilon_{2pq}}}\rangle.$$
		
		Furthermore, we have $h_2(2p)=8$, $h_2(pq)=4$,  $h_2(2pq)=4$ and  $h_2(\KK)=16$.
		\item Consider the prime numbers $p=17$ and  $q=191$. In this case, the   conditions of the second item of the above theorem are verified and we have:
			$$E_{\KK}=\langle -1,  \varepsilon_{2}, \varepsilon_{p},   \sqrt{\varepsilon_{q}}, \sqrt{\varepsilon_{2q}},  \sqrt{\varepsilon_{pq}} , \sqrt{\varepsilon_{2}^a \varepsilon_{p}^a  \sqrt{\varepsilon_{q}}  \sqrt{\varepsilon_{pq}} 
			\sqrt{\varepsilon_{2p}} },
		\sqrt{\varepsilon_{2}^a  \varepsilon_{p}^a   \sqrt{\varepsilon_{2q}}  \sqrt{\varepsilon_{2pq}} 
			\sqrt{\varepsilon_{2p}} }   \rangle,$$
	  for some 	$a\in\{0,1\}$ such that $a \equiv 1+u\pmod2$.	Furthermore, we have $h_2(2p)=2$, $h_2(pq)=8$,  $h_2(2pq)=4$ and  $h_2(\KK)=16$.
	\end{enumerate}
	\end{exam}

\begin{theorem}\label{T_7_1_C2} Let $p\equiv 1\pmod{8}$ and $q\equiv7\pmod 8$ be two primes such that     $\genfrac(){}{0}{p}{q} =1$. 
	Put     $\KK=\QQ(\sqrt 2, \sqrt{p}, \sqrt{q} )$.  Assume furthermore that $x+1$ and $p(v+1)$ are squares in $\mathbb{N}$, where $x$ and $v$ are defined in Lemma \ref{lm expressions of units under cond 3 case: 7=1 mod 8 and ()=1}.
	We have
	\begin{enumerate}[\rm 1)]
		\item If $N(\varepsilon_{2p})=-1$, then
		\begin{enumerate}[\rm $\bullet$]
			
			\item The unit group of $\KK$ is :
			$$E_{\KK}=\langle -1,   \varepsilon_{2}, \varepsilon_{p} ,    \sqrt{\varepsilon_{q}}, \sqrt{\varepsilon_{2q}},\sqrt{ \varepsilon_{pq}} ,\sqrt{ \varepsilon_{2pq}}, \sqrt{\varepsilon_{2}\varepsilon_{p}\varepsilon_{2p}} \rangle.$$
			\item  The $2$-class number  of  $\KK$ equals $ \frac{1}{2^{4}} h_2(2p)h_2(pq)h_2(2pq)$.  
		\end{enumerate}
		\item If $N(\varepsilon_{2p})=1$ and   $a\in\{0,1\}$ such that $a \equiv 1+u\pmod2$, then
		
		\begin{enumerate}[\rm $\bullet$]
			\item  The unit group of $\KK$ is :
			$$E_{\KK}=\langle -1,   \varepsilon_{2}, \varepsilon_{p},   \sqrt{\varepsilon_{q}}, \sqrt{\varepsilon_{2q}},  \sqrt{\varepsilon_{pq}} ,\sqrt{ \varepsilon_{2pq}},    \sqrt{\varepsilon_{2}^{a\alpha}\varepsilon_{p}^{u\alpha}     \sqrt{\varepsilon_{2q}}^\alpha\sqrt{\varepsilon_{2pq}}^\alpha
				\sqrt{\varepsilon_{2p}}^{1+\gamma} }  \rangle$$
			where $\alpha$, $\gamma\in \{0,1\}$ such that $\alpha\not=\gamma$ and $\alpha =1$ if and only if  $\varepsilon_{2}^a \varepsilon_{p}^u      \sqrt{\varepsilon_{2q}} \sqrt{\varepsilon_{2pq}}\sqrt{\varepsilon_{2p}}$ is a square in $\KK$.
			\item  The $2$-class number  of  $\KK$ equals $ \frac{1}{2^{4-\alpha}} h_2(2p)h_2(pq)h_2(2pq)$.
		\end{enumerate}
	\end{enumerate}
\end{theorem}
\begin{proof} 
		 	\begin{enumerate}[\rm 1)]
		\item  
		Assume that   $N(\varepsilon_{2p})=-1$.  By Lemma \ref{units of k1},     $\{\varepsilon_{2}, \varepsilon_{p},	\sqrt{\varepsilon_{2}\varepsilon_{p}\varepsilon_{2p}}\}$  is a  fundamental system of units of $k_1$. Using Lemmas \ref{lm expressions of units q_2q} and \ref{lm expressions of units under cond 3 case: 7=1 mod 8 and ()=1}, we check   that $\{\varepsilon_{2}, \sqrt{\varepsilon_{q}}, \sqrt{\varepsilon_{2q}}\}$ and $\{ \varepsilon_{2}, 	{\varepsilon_{pq}},\sqrt{\varepsilon_{2pq}}\}$ are respectively fundamental systems of units of $k_2$ and $k_3$.
		It follows that,  	$$E_{k_1}E_{k_2}E_{k_3}=\langle-1,  \varepsilon_{2}, \varepsilon_{p},   \sqrt{\varepsilon_{q}}, \sqrt{\varepsilon_{2q}},   {\varepsilon_{pq}} ,\sqrt{ \varepsilon_{2pq}}, \sqrt{\varepsilon_{2}\varepsilon_{p}\varepsilon_{2p}}\rangle.$$	
		Thus we shall determine elements of $E_{k_1}E_{k_2}E_{k_3}$ which are squares in $\KK$. 
		Notice that by Lemma \ref{lm expressions of units under cond 3 case: 7=1 mod 8 and ()=1},
		$\varepsilon_{pq}$ is a square in  $\KK$.
	Let  $\xi$ is an element of $\KK$ which is the  square root of an element of $E_{k_1}E_{k_2}E_{k_3}$. We can assume that
		$$\xi^2=\varepsilon_{2}^a\varepsilon_{p}^b \sqrt{\varepsilon_{q}}^c\sqrt{\varepsilon_{2q}}^d\sqrt{\varepsilon_{2pq}}^f
		\sqrt{\varepsilon_{2}\varepsilon_{p}\varepsilon_{2p}}^g,$$
		where $a, b, c, d,   f$ and $g$ are in $\{0, 1\}$.

		\noindent\ding{224}  Let us start	by applying   the norm map $N_{\KK/k_2}=1+\tau_2$. We have 
		$\sqrt{\varepsilon_{2}\varepsilon_{p}\varepsilon_{2p}}^{1+ \tau_2}=(-1)^v\varepsilon_{2}$, for some $v\in\{0,1\}$. By means of the table
		  \eqref{norm q=7 p=1mod4q} and Table \ref{tab3 case: 7=1 mod 8 and ()=1},   we get:
		%	\begin{eqnarray}\label{T_7_-1_tau2_N=-1}
		%	\begin{tabular}{ |c|c|c|c|c|c|c|c|c|}
		%		\hline
		%		$\varepsilon$&$\varepsilon_{2}$ &$ {\varepsilon_{p}}$&$\sqrt{\varepsilon_{q}}$&$\sqrt{\varepsilon_{2q}}$&$\sqrt{\varepsilon_{pq}}$&$\sqrt{\varepsilon_{2pq}}$&$\sqrt{\varepsilon_{2}\varepsilon_{p}\varepsilon_{2p}}$ \\
		%		\hline
		%		$\varepsilon^{1+\tau_2}$&$\varepsilon_{2}^2$ &$-1$&	$\varepsilon_{q}$&$\varepsilon_{2q}$&$ 1$& $1$&$(-1)^v\varepsilon_{2}$ \\
		%		\hline
		%	\end{tabular} 
		%\end{eqnarray}
		\begin{eqnarray*}
			N_{\KK/k_2}(\xi^2)&=&
			\varepsilon_{2}^{2a}\cdot(-1)^b \cdot \varepsilon_{q}^c\cdot\varepsilon_{2q}^d\cdot1 \cdot (-1)^{gv}\varepsilon_{2}^{g}\\
			&=&	\varepsilon_{2}^{2a}  \varepsilon_{q}^c\varepsilon_{2q}^d\cdot(-1)^{b+gv} \varepsilon_{2}^g.
		\end{eqnarray*}
		
		Thus    $b+gv\equiv 0\pmod2$ and $g=0$. Therefore,  $b=0$ and
		$$\xi^2=\varepsilon_{2}^a  \sqrt{\varepsilon_{q}}^c\sqrt{\varepsilon_{2q}}^d \sqrt{\varepsilon_{2pq}}^f.$$

		\noindent\ding{224} Let us apply the norm $N_{\KK/k_5}=1+\tau_1\tau_2$, with $k_5=\QQ(\sqrt{q}, \sqrt{2p})$. By the table  \eqref{norm q=7 p=1mod4q} and Table \ref{tab3 case: 7=1 mod 8 and ()=1}, we  have:
		%	\begin{eqnarray}\label{T_7_-1_tau1tau2_N=-1}
		%	\begin{tabular}{ |c|c|c|c|c|c|c|c|c|}
		%		\hline
		%		$\varepsilon$&$\varepsilon_{2}$ &$ {\varepsilon_{p}}$&$\sqrt{\varepsilon_{q}}$&$\sqrt{\varepsilon_{2q}}$&$\sqrt{\varepsilon_{pq}}$&$\sqrt{\varepsilon_{2pq}}$ \\
		%		\hline
		%		$\varepsilon^{1+\tau_1\tau_2}$&$-1$ &$-1$&	$-\varepsilon_{q}$&$-1$&$ -1$& ${-\varepsilon_{2pq}}$   \\
		%		\hline
		%	\end{tabular}
		%	\end{eqnarray}	
		\begin{eqnarray*}
			N_{\KK/k_5}(\xi^2)&=&(-1)^a\cdot (-1)^c\cdot \varepsilon_{  q}^c\cdot(-1)^d \cdot (-1)^f\cdot \varepsilon_{  2pq}^f\\
			&=&	 (-1)^{a+c+d +f}\varepsilon_{  q}^c\cdot \varepsilon_{  2pq}^f.
		\end{eqnarray*}
		Thus $a+c+d +f=0\pmod 2$ and $f=c$. Thus, $a=d $. Therefore, 
		$$\xi^2=\varepsilon_{2}^a  \sqrt{\varepsilon_{q}}^c\sqrt{\varepsilon_{2q}}^a \sqrt{\varepsilon_{2pq}}^c.$$

		\noindent\ding{224} Let us apply the norm $N_{\KK/k_6}=1+\tau_1\tau_3$, with $k_6=\QQ(\sqrt{p}, \sqrt{2q})$. By the table \eqref{norm q=7 p=1mod4q} and Table \ref{tab3 case: 7=1 mod 8 and ()=1}, we  have:
		%	\begin{eqnarray}\label{T_7_-1_tau1tau3_N=-1}
		%		\begin{tabular}{ |c|c|c|c|c|c|c|c|c|}
		%			\hline
		%			$\varepsilon$&$\varepsilon_{2}$ &$ {\varepsilon_{p}}$&$\sqrt{\varepsilon_{q}}$&$\sqrt{\varepsilon_{2q}}$&$\sqrt{\varepsilon_{pq}}$&$\sqrt{\varepsilon_{2pq}}$ \\
		%			\hline
		%			$\varepsilon^{1+\tau_1\tau_3}$&$-1$ &$\varepsilon_{p}^2$&	$-1$&$-\varepsilon_{2q}$&$ -1$& ${-\varepsilon_{2pq}}$   \\
		%			\hline
		%		\end{tabular} 
		%	\end{eqnarray}
		\begin{eqnarray*}
			N_{\KK/k_6}(\xi^2)&=&(-1)^a\cdot (-1)^c\cdot (-1)^a\cdot \varepsilon_{2  q}^a\cdot  (-1)^c\cdot \varepsilon_{  2pq}^c\\
			&=&	  \varepsilon_{2  q}^a\varepsilon_{  2pq}^c.
		\end{eqnarray*}
		Thus 
		$a=c$. Therefore   
		$$\xi^2=\varepsilon_{2}^a  \sqrt{\varepsilon_{q}}^a\sqrt{\varepsilon_{2q}}^a \sqrt{\varepsilon_{2pq}}^a.$$

		\noindent\ding{224} Let us apply the norm $N_{\KK/k_4}=1+\tau_1$, with $k_4=\QQ(\sqrt{p}, \sqrt{q})$.  By the table \eqref{norm q=7 p=1mod4q} and Table \ref{tab3 case: 7=1 mod 8 and ()=1}, we  have:		
		\begin{eqnarray*}
			N_{\KK/k_4}(\xi^2)&=&(-1)^a\cdot (-1)^a\cdot \varepsilon_{  q}^a\cdot (-1)^a \cdot(-1)^a  \\
			&=&	 \varepsilon_{  q}^a.
		\end{eqnarray*}
		Thus   $a=0$.
		  It follows that the only element of $E_{k_1}E_{k_2}E_{k_3}$ that is a square in $\KK$ is $\varepsilon_{pq}$. So the first item.

	\item Assume that   $N(\varepsilon_{2p})=1$.  By Lemma \ref{units of k1},     $\{\varepsilon_{2}, \varepsilon_{p},	\sqrt{\varepsilon_{2p}}\}$  is a  fundamental system of units of $k_1$. Using Lemmas \ref{lm expressions of units q_2q} and \ref{lm expressions of units under cond 3 case: 7=1 mod 8 and ()=1}, we check   that $\{\varepsilon_{2}, \sqrt{\varepsilon_{q}}, \sqrt{\varepsilon_{2q}}\}$ and $\{ \varepsilon_{2}, 	{\varepsilon_{pq}},\sqrt{\varepsilon_{2pq}}\}$ are respectively fundamental systems of units of $k_2$ and $k_3$.
	It follows that,  	$$E_{k_1}E_{k_2}E_{k_3}=\langle-1,  \varepsilon_{2}, \varepsilon_{p},   \sqrt{\varepsilon_{q}}, \sqrt{\varepsilon_{2q}},   {\varepsilon_{pq}} ,\sqrt{ \varepsilon_{2pq}}, \sqrt{\varepsilon_{2p}}\rangle.$$	
	Thus we shall determine elements of $E_{k_1}E_{k_2}E_{k_3}$ which are squares in $\KK$. 
	Note that  
	$\varepsilon_{pq}$ is a square in  $\KK$. 	Let  $\xi$ is an element of $\KK$ which is the  square root of an element of $E_{k_1}E_{k_2}E_{k_3}$. Assume that
	$$\xi^2=\varepsilon_{2}^a\varepsilon_{p}^b \sqrt{\varepsilon_{q}}^c\sqrt{\varepsilon_{2q}}^d\sqrt{\varepsilon_{2pq}}^f
	\sqrt{\varepsilon_{2p}}^g,$$
	where $a, b, c, d,   f$ and $g$ are in $\{0, 1\}$.

\noindent\ding{224}  Let us start	by applying   the norm map $N_{\KK/k_2}=1+\tau_2$.   By means of 
\eqref{T_3_-1_eqi2p_N=1}, \eqref{norm q=7 p=1mod4q} and Table \ref{tab3 case: 7=1 mod 8 and ()=1},   we get:
\begin{eqnarray*}
	N_{\KK/k_2}(\xi^2)&=&
	\varepsilon_{2}^{2a}\cdot(-1)^b \cdot \varepsilon_{q}^c\cdot\varepsilon_{2q}^d\cdot1 \cdot (-1)^{gu} \\
	&=&	\varepsilon_{2}^{2a}  \varepsilon_{q}^c\varepsilon_{2q}^d\cdot(-1)^{b+gu}  .
\end{eqnarray*}

Thus    $b+gu\equiv 0\pmod2$.

\noindent\ding{224} Let us apply the norm $N_{\KK/k_5}=1+\tau_1\tau_2$, with $k_5=\QQ(\sqrt{q}, \sqrt{2p})$. By \eqref{T_3_-1_eqi2p_N=1}, \eqref{norm q=7 p=1mod4q} and Table \ref{tab3 case: 7=1 mod 8 and ()=1}, we  have:
\begin{eqnarray*}
	N_{\KK/k_5}(\xi^2)&=&(-1)^a\cdot(-1)^b\cdot (-1)^c\cdot \varepsilon_{  q}^c\cdot(-1)^d \cdot (-1)^f\cdot \varepsilon_{  2pq}^f\cdot (-1)^g\cdot \varepsilon_{  2p}^g\\
	&=&	 (-1)^{a+b+c+d +f+g}\varepsilon_{  q}^c\cdot \varepsilon_{  2pq}^f\cdot \varepsilon_{  2p}^g.
\end{eqnarray*}
Thus $a+b+c+d +f+g=0\pmod 2$ and $ c  +f+g=0\pmod 2$. So $a+b +d =0\pmod 2$.

\noindent\ding{224} Let us apply the norm $N_{\KK/k_6}=1+\tau_1\tau_3$, with $k_6=\QQ(\sqrt{p}, \sqrt{2q})$. By \eqref{T_3_-1_eqi2p_N=1}, \eqref{norm q=7 p=1mod4q} and Table \ref{tab3 case: 7=1 mod 8 and ()=1}, we  have:
%	\begin{eqnarray}\label{T_7_-1_tau1tau3_N=-1}
%		\begin{tabular}{ |c|c|c|c|c|c|c|c|c|}
%			\hline
%			$\varepsilon$&$\varepsilon_{2}$ &$ {\varepsilon_{p}}$&$\sqrt{\varepsilon_{q}}$&$\sqrt{\varepsilon_{2q}}$&$\sqrt{\varepsilon_{pq}}$&$\sqrt{\varepsilon_{2pq}}$ \\
%			\hline
%			$\varepsilon^{1+\tau_1\tau_3}$&$-1$ &$\varepsilon_{p}^2$&	$-1$&$-\varepsilon_{2q}$&$ -1$& ${-\varepsilon_{2pq}}$   \\
%			\hline
%		\end{tabular} 
%	\end{eqnarray}
\begin{eqnarray*}
	N_{\KK/k_6}(\xi^2)&=&(-1)^a\cdot \varepsilon_{p}^{2b}\cdot (-1)^c \cdot (-1)^d\cdot \varepsilon_{2q}^{d} \cdot (-1)^f\cdot \varepsilon_{  2pq}^f\cdot (-1)^{gu+g}  \\
	&=&	 \varepsilon_{p}^{2b} (-1)^{a+c+d+f+gu+g} \varepsilon_{2q}^{d}\cdot\varepsilon_{  2pq}^f
\end{eqnarray*}
Thus 
  $a+c+d+f+gu+g=0\pmod 2$ and   $d=f$. So  $a+c+gu+g=0\pmod 2$. Therefore,
$$\xi^2=\varepsilon_{2}^a\varepsilon_{p}^b \sqrt{\varepsilon_{q}}^c\sqrt{\varepsilon_{2q}}^d\sqrt{\varepsilon_{2pq}}^d
\sqrt{\varepsilon_{2p}}^g,$$

\noindent\ding{224} Let us apply the norm $N_{\KK/k_4}=1+\tau_1$, with $k_4=\QQ(\sqrt{p}, \sqrt{q})$.  By \eqref{T_3_-1_eqi2p_N=1}, \eqref{norm q=7 p=1mod4q} and Table \ref{tab3 case: 7=1 mod 8 and ()=1}, we  have:
\begin{eqnarray*}
	N_{\KK/k_4}(\xi^2)&=&(-1)^a\cdot \varepsilon_{p}^{2b}\cdot (-1)^c\cdot \varepsilon_{ q}^{c}\cdot (-1)^d \cdot (-1)^d\cdot  (-1)^{gu+g}  \\
	&=&	 \varepsilon_{p}^{2b} (-1)^{a+c+gu+g} \varepsilon_{ q}^{c} 
\end{eqnarray*}
Thus 
$c=0$ and so $a+gu+g=0\pmod 2$. Since $c+f+g=0\pmod 2$, this implies that $g=f=d$. As $b+gu=0=b+du\pmod 2$, then we have:
$$\xi^2=\varepsilon_{2}^a\varepsilon_{p}^{du}  \sqrt{\varepsilon_{2q}}^d\sqrt{\varepsilon_{2pq}}^d
\sqrt{\varepsilon_{2p}}^d,$$
where $a+du+d=0\pmod 2$. So we have the second item.
\end{enumerate}

\end{proof}

 \begin{exam}Using  PARI/GP calculator version  2.15.5 (64bit), Feb 11 2024, we  get the following examples which illustrate the above theorem.
 	\begin{enumerate}[\rm $1)$]
 		%	\item Let $p=463$ and  $q=113$. We have   $p\equiv 1\pmod{8}$ and $q\equiv7\pmod 8$,     $\genfrac(){}{0}{p}{q} =1$,  
 		%	$x+1$ and $v+1$ are squares in $\mathbb{N}$, where $x$ and $v$ are defined in Lemma \ref{lm expressions of units under cond 3 case: 7=1 mod 8 and ()=1} and   $N(\varepsilon_{2p})=-1$. In this case we have :
 		%	$$E_{\KK}=\langle -1,  \varepsilon_{2}, \varepsilon_{p},   \sqrt{\varepsilon_{q}}, \sqrt{\varepsilon_{2q}},  \sqrt{\varepsilon_{pq}} , \sqrt{\varepsilon_{2}\varepsilon_{p}\varepsilon_{2p}}, 
 		%	\sqrt{\sqrt{\varepsilon_{q}}  \sqrt{\varepsilon_{2q}}  \sqrt{\varepsilon_{pq}}  \sqrt{\varepsilon_{2pq}}}\rangle.$$
 		\item Consider the prime numbers $p=41$ and  $q=431$. In this case, the   conditions of the first item of the above theorem are verified and we have:
 	$$E_{\KK}=\langle -1,   \varepsilon_{2}, \varepsilon_{p} ,    \sqrt{\varepsilon_{q}}, \sqrt{\varepsilon_{2q}},\sqrt{ \varepsilon_{pq}} ,\sqrt{ \varepsilon_{2pq}}, \sqrt{\varepsilon_{2}\varepsilon_{p}\varepsilon_{2p}} \rangle.$$
 		
 		Furthermore, we have $h_2(2p)=4$, $h_2(pq)=4$,  $h_2(2pq)=8$ and  $h_2(\KK)=8$.
 		\item  Consider the prime numbers $p=17$ and  $q=47$. In this case, the   conditions of the second item of the above theorem are verified and we have:
 		$$E_{\KK}=\langle -1,   \varepsilon_{2}, \varepsilon_{p},   \sqrt{\varepsilon_{q}}, \sqrt{\varepsilon_{2q}},  \sqrt{\varepsilon_{pq}} ,\sqrt{ \varepsilon_{2pq}},    \sqrt{\varepsilon_{2}^{a}\varepsilon_{p}^{u}     \sqrt{\varepsilon_{2q}}\sqrt{\varepsilon_{2pq}}
 			\sqrt{\varepsilon_{2p}} }  \rangle,$$
 for some 	$a\in\{0,1\}$ such that $a \equiv 1+u\pmod2$.	Furthermore, we have $h_2(2p)=2$, $h_2(pq)=8$,  $h_2(2pq)=4$ and  $h_2(\KK)=8$.
 	\end{enumerate}
 \end{exam}

	\begin{theorem}\label{T_7_1-C3} Let $p\equiv 1\pmod{8}$ and $q\equiv7\pmod 8$ be two primes such that     $\genfrac(){}{0}{p}{q} =1$. 
		Put     $\KK=\QQ(\sqrt 2, \sqrt{p}, \sqrt{q} )$.  Assume furthermore that $x+1$ and $2p(v+1)$ are squares in $\mathbb{N}$, where $x$ and $v$ are defined in Lemma \ref{lm expressions of units under cond 3 case: 7=1 mod 8 and ()=1}.
		\begin{enumerate}[\rm 1)]
			\item Assume that $N(\varepsilon_{2p})=-1$.  We have
			\begin{enumerate}[\rm $\bullet$]
				
				\item The unit group of $\KK$ is :
				$$E_{\KK}=\langle -1,   \varepsilon_{2}, \varepsilon_{p} ,    \sqrt{\varepsilon_{q}}, \sqrt{\varepsilon_{2q}},\sqrt{ \varepsilon_{pq}} ,\sqrt{ \varepsilon_{2pq}}, \sqrt{\varepsilon_{2}\varepsilon_{p}\varepsilon_{2p}} \rangle.$$
				\item  The $2$-class number  of  $\KK$ equals $ \frac{1}{2^{4}} h_2(2p)h_2(pq)h_2(2pq)$.  
			\end{enumerate}
			\item Assume that $N(\varepsilon_{2p})=1$ and let $a\in\{0,1\}$ such that $a \equiv 1+u\pmod2$.   We have
			
			\begin{enumerate}[\rm $\bullet$]
				\item  The unit group of $\KK$ is :
				$$E_{\KK}=\langle -1,   \varepsilon_{2}, \varepsilon_{p},   \sqrt{\varepsilon_{q}}, \sqrt{\varepsilon_{2q}},  \sqrt{\varepsilon_{pq}} ,\sqrt{ \varepsilon_{2pq}},    \sqrt{\varepsilon_{2}^{a\alpha}\varepsilon_{p}^{u\alpha}     \sqrt{\varepsilon_{2q}}^\alpha\sqrt{\varepsilon_{2pq}}^\alpha
					\sqrt{\varepsilon_{2p}}^{1+\gamma} }  \rangle$$
				where $\alpha$, $\gamma\in \{0,1\}$ such that $\alpha\not=\gamma$ and $\alpha =1$ if and only if $\varepsilon_{2}^a \varepsilon_{p}^a      \sqrt{\varepsilon_{2q}} \sqrt{\varepsilon_{2pq}}\sqrt{\varepsilon_{2p}}$ is a square in $\KK$.
				\item  The $2$-class number  of  $\KK$ equals $ \frac{1}{2^{4-\alpha}} h_2(2p)h_2(pq)h_2(2pq)$.
			\end{enumerate}
		\end{enumerate}
	\end{theorem}
	\begin{proof} 
	 	\begin{enumerate}[\rm 1)]
			\item 	Assume that   $N(\varepsilon_{2p})=-1$.  By Lemma \ref{units of k1},     $\{\varepsilon_{2}, \varepsilon_{p},	\sqrt{\varepsilon_{2}\varepsilon_{p}\varepsilon_{2p}}\}$  is a  fundamental system of units of $k_1$. Using Lemmas \ref{lm expressions of units q_2q} and \ref{lm expressions of units under cond 3 case: 7=1 mod 8 and ()=1}, we check   that $\{\varepsilon_{2}, \sqrt{\varepsilon_{q}}, \sqrt{\varepsilon_{2q}}\}$ and $\{ \varepsilon_{2}, 	{\varepsilon_{pq}},\sqrt{\varepsilon_{2pq}}\}$ are respectively fundamental systems of units of $k_2$ and $k_3$.
			It follows that,  	$$E_{k_1}E_{k_2}E_{k_3}=\langle-1,  \varepsilon_{2}, \varepsilon_{p},   \sqrt{\varepsilon_{q}}, \sqrt{\varepsilon_{2q}},   {\varepsilon_{pq}} ,\sqrt{ \varepsilon_{2pq}}, \sqrt{\varepsilon_{2}\varepsilon_{p}\varepsilon_{2p}}\rangle.$$	
			Thus we shall determine elements of $E_{k_1}E_{k_2}E_{k_3}$ which are squares in $\KK$. 
			Notice that by Lemma \ref{lm expressions of units under cond 3 case: 7=1 mod 8 and ()=1},
			$\varepsilon_{pq}$ is a square in  $\KK$.
			Let  $\xi$ is an element of $\KK$ which is the  square root of an element of $E_{k_1}E_{k_2}E_{k_3}$. We can assume that
			$$\xi^2=\varepsilon_{2}^a\varepsilon_{p}^b \sqrt{\varepsilon_{q}}^c\sqrt{\varepsilon_{2q}}^d\sqrt{\varepsilon_{2pq}}^f
			\sqrt{\varepsilon_{2}\varepsilon_{p}\varepsilon_{2p}}^g,$$
			where $a, b, c, d,   f$ and $g$ are in $\{0, 1\}$.
			
			\noindent\ding{224}  Let us start	by applying   the norm map $N_{\KK/k_2}=1+\tau_2$. We have 
		$\sqrt{\varepsilon_{2}\varepsilon_{p}\varepsilon_{2p}}^{1+ \tau_2}=(-1)^v\varepsilon_{2}$, for some $v\in\{0,1\}$. By means of 
		\eqref{norm q=7 p=1mod4q} and Table \ref{tab3 case: 7=1 mod 8 and ()=1},   we get:
		%	\begin{eqnarray}\label{T_7_-1_tau2_N=-1}
		%	\begin{tabular}{ |c|c|c|c|c|c|c|c|c|}
		%		\hline
		%		$\varepsilon$&$\varepsilon_{2}$ &$ {\varepsilon_{p}}$&$\sqrt{\varepsilon_{q}}$&$\sqrt{\varepsilon_{2q}}$&$\sqrt{\varepsilon_{pq}}$&$\sqrt{\varepsilon_{2pq}}$&$\sqrt{\varepsilon_{2}\varepsilon_{p}\varepsilon_{2p}}$ \\
		%		\hline
		%		$\varepsilon^{1+\tau_2}$&$\varepsilon_{2}^2$ &$-1$&	$\varepsilon_{q}$&$\varepsilon_{2q}$&$ 1$& $1$&$(-1)^v\varepsilon_{2}$ \\
		%		\hline
		%	\end{tabular} 
		%\end{eqnarray}
		\begin{eqnarray*}
			N_{\KK/k_2}(\xi^2)&=&
			\varepsilon_{2}^{2a}\cdot(-1)^b \cdot \varepsilon_{q}^c\cdot\varepsilon_{2q}^d\cdot1 \cdot (-1)^{gv}\varepsilon_{2}^{g}\\
			&=&	\varepsilon_{2}^{2a}  \varepsilon_{q}^c\varepsilon_{2q}^d\cdot(-1)^{b+gv} \varepsilon_{2}^g.
		\end{eqnarray*}
		
		Thus    $b+gv\equiv 0\pmod2$ and $g=0$. Therefore,  $b=0$ and
		$$\xi^2=\varepsilon_{2}^a  \sqrt{\varepsilon_{q}}^c\sqrt{\varepsilon_{2q}}^d \sqrt{\varepsilon_{2pq}}^f.$$	
			
			\noindent\ding{224} Let us apply the norm $N_{\KK/k_5}=1+\tau_1\tau_2$, with $k_5=\QQ(\sqrt{q}, \sqrt{2p})$. By the table \eqref{norm q=7 p=1mod4q} and Table \ref{tab3 case: 7=1 mod 8 and ()=1}, we  have:
		%	\begin{eqnarray}\label{T_7_-1_tau1tau2_N=-1}
		%	\begin{tabular}{ |c|c|c|c|c|c|c|c|c|}
		%		\hline
		%		$\varepsilon$&$\varepsilon_{2}$ &$ {\varepsilon_{p}}$&$\sqrt{\varepsilon_{q}}$&$\sqrt{\varepsilon_{2q}}$&$\sqrt{\varepsilon_{pq}}$&$\sqrt{\varepsilon_{2pq}}$ \\
		%		\hline
		%		$\varepsilon^{1+\tau_1\tau_2}$&$-1$ &$-1$&	$-\varepsilon_{q}$&$-1$&$ -1$& ${-\varepsilon_{2pq}}$   \\
		%		\hline
		%	\end{tabular}
		%	\end{eqnarray}	
		\begin{eqnarray*}
			N_{\KK/k_5}(\xi^2)&=&(-1)^a\cdot (-1)^c\cdot \varepsilon_{  q}^c\cdot(-1)^d \cdot (-1)^f\cdot \varepsilon_{  2pq}^f\\
			&=&	 (-1)^{a+c+d +f}\varepsilon_{  q}^c\cdot \varepsilon_{  2pq}^f.
		\end{eqnarray*}
		Thus $a+c+d +f=0\pmod 2$ and $f=c$. Thus, $a=d $. Therefore, 
		$$\xi^2=\varepsilon_{2}^a  \sqrt{\varepsilon_{q}}^c\sqrt{\varepsilon_{2q}}^a \sqrt{\varepsilon_{2pq}}^c.$$

		\noindent\ding{224} Let us apply the norm $N_{\KK/k_6}=1+\tau_1\tau_3$, with $k_6=\QQ(\sqrt{p}, \sqrt{2q})$. By the table \eqref{norm q=7 p=1mod4q} and Table \ref{tab3 case: 7=1 mod 8 and ()=1}, we  have:
		%	\begin{eqnarray}\label{T_7_-1_tau1tau3_N=-1}
		%		\begin{tabular}{ |c|c|c|c|c|c|c|c|c|}
		%			\hline
		%			$\varepsilon$&$\varepsilon_{2}$ &$ {\varepsilon_{p}}$&$\sqrt{\varepsilon_{q}}$&$\sqrt{\varepsilon_{2q}}$&$\sqrt{\varepsilon_{pq}}$&$\sqrt{\varepsilon_{2pq}}$ \\
		%			\hline
		%			$\varepsilon^{1+\tau_1\tau_3}$&$-1$ &$\varepsilon_{p}^2$&	$-1$&$-\varepsilon_{2q}$&$ -1$& ${-\varepsilon_{2pq}}$   \\
		%			\hline
		%		\end{tabular} 
		%	\end{eqnarray}
		\begin{eqnarray*}
			N_{\KK/k_6}(\xi^2)&=&(-1)^a\cdot (-1)^c\cdot (-1)^a\cdot \varepsilon_{2  q}^a\cdot  (-1)^c\cdot \varepsilon_{  2pq}^c\\
			&=&	  \varepsilon_{2  q}^a\varepsilon_{  2pq}^c.
		\end{eqnarray*}
		Thus 
		$a=c$. Therefore   
		$$\xi^2=\varepsilon_{2}^a  \sqrt{\varepsilon_{q}}^a\sqrt{\varepsilon_{2q}}^a \sqrt{\varepsilon_{2pq}}^a.$$

	\item For the case $N(\varepsilon_{2p})=1$, we check the result by  the same computations as in the previous theorem.

		\end{enumerate}
	\end{proof}

	 \begin{exam}Using  PARI/GP calculator version  2.15.5 (64bit), Feb 11 2024, we  get the following examples which illustrate the above theorem.
		\begin{enumerate}[\rm $1)$]
			\item Consider the prime numbers $p=41$ and  $q=23$. In this case, the   conditions of the first item of the above theorem are verified and we have:
			$$E_{\KK}=\langle -1,   \varepsilon_{2}, \varepsilon_{p} ,    \sqrt{\varepsilon_{q}}, \sqrt{\varepsilon_{2q}},\sqrt{ \varepsilon_{pq}} ,\sqrt{ \varepsilon_{2pq}}, \sqrt{\varepsilon_{2}\varepsilon_{p}\varepsilon_{2p}} \rangle.$$
			
			Furthermore, we have $h_2(2p)=4$, $h_2(pq)=4$,  $h_2(2pq)=4$ and  $h_2(\KK)=4$.
			\item  Consider the prime numbers $p=17$ and  $q=239$. In this case, the   conditions of the second item of the above theorem are verified and we have:
			$$E_{\KK}=\langle -1,   \varepsilon_{2}, \varepsilon_{p},   \sqrt{\varepsilon_{q}}, \sqrt{\varepsilon_{2q}},  \sqrt{\varepsilon_{pq}} ,\sqrt{ \varepsilon_{2pq}},    \sqrt{\varepsilon_{2}^{a}\varepsilon_{p}^{u}     \sqrt{\varepsilon_{2q}}\sqrt{\varepsilon_{2pq}}
				\sqrt{\varepsilon_{2p}} }  \rangle,$$
			for some 	$a\in\{0,1\}$ such that $a \equiv 1+u\pmod2$.	Furthermore, we have $h_2(2p)=2$, $h_2(pq)=8$,  $h_2(2pq)=4$ and  $h_2(\KK)=8$.
		\end{enumerate}
	\end{exam}

	\begin{theorem}\label{T_7_1-C4} Let $p\equiv 1\pmod{8}$ and $q\equiv7\pmod 8$ be two primes such that     $\genfrac(){}{0}{p}{q} =1$. 
		Put     $\KK=\QQ(\sqrt 2, \sqrt{p}, \sqrt{q} )$.  Assume furthermore that $p(x+1)$ and $v+1$ are squares in $\mathbb{N}$, where $x$ and $v$ are defined in Lemma \ref{lm expressions of units under cond 3 case: 7=1 mod 8 and ()=1}.
		\begin{enumerate}[\rm 1)]
			\item Assume that $N(\varepsilon_{2p})=-1$.  We have
			\begin{enumerate}[\rm $\bullet$]
				
				\item The unit group of $\KK$ is :
				$$E_{\KK}=\langle -1,   \varepsilon_{2}, \varepsilon_{p} ,    \sqrt{\varepsilon_{q}}, \sqrt{\varepsilon_{2q}},\sqrt{ \varepsilon_{pq}} ,\sqrt{ \varepsilon_{2pq}}, \sqrt{\varepsilon_{2}\varepsilon_{p}\varepsilon_{2p}} \rangle.$$
				\item  The $2$-class number  of  $\KK$ equals $ \frac{1}{2^{4}} h_2(2p)h_2(pq)h_2(2pq)$.  
			\end{enumerate}
			\item Assume that $N(\varepsilon_{2p})=1$ and let $a\in\{0,1\}$ such that $a \equiv 1+u\pmod2$.   We have
			
			\begin{enumerate}[\rm $\bullet$]
				\item  The unit group of $\KK$ is :
				$$E_{\KK}=\langle -1,   \varepsilon_{2}, \varepsilon_{p},   \sqrt{\varepsilon_{q}}, \sqrt{\varepsilon_{2q}},  \sqrt{\varepsilon_{pq}} ,\sqrt{ \varepsilon_{2pq}},    \sqrt{\varepsilon_{2}^{a\alpha}\varepsilon_{p}^{u\alpha}     \sqrt{\varepsilon_{ q}}^\alpha\sqrt{\varepsilon_{ pq}}^\alpha
					\sqrt{\varepsilon_{2p}}^{1+\gamma} }  \rangle$$
				where $\alpha$, $\gamma\in \{0,1\}$ such that $\alpha\not=\gamma$ and $\alpha =1$ if and only if $\varepsilon_{2}^a \varepsilon_{p}^u      \sqrt{\varepsilon_{ q}} \sqrt{\varepsilon_{ pq}}\sqrt{\varepsilon_{2p}}$ is a square in $\KK$.
				\item  The $2$-class number  of  $\KK$ equals $ \frac{1}{2^{4-\alpha}} h_2(2p)h_2(pq)h_2(2pq)$.
			\end{enumerate}
		\end{enumerate}
	\end{theorem}
	\begin{proof} 
	 	\begin{enumerate}[\rm 1)]
			\item 	Assume that   $N(\varepsilon_{2p})=-1$.  By Lemma \ref{units of k1},     $\{\varepsilon_{2}, \varepsilon_{p},	\sqrt{\varepsilon_{2}\varepsilon_{p}\varepsilon_{2p}}\}$  is a  fundamental system of units of $k_1$. Using Lemmas \ref{lm expressions of units q_2q} and \ref{lm expressions of units under cond 3 case: 7=1 mod 8 and ()=1}, we check   that $\{\varepsilon_{2}, \sqrt{\varepsilon_{q}}, \sqrt{\varepsilon_{2q}}\}$ and $\{ \varepsilon_{2}, 	{\varepsilon_{2pq}},\sqrt{\varepsilon_{pq}}\}$ are respectively fundamental systems of units of $k_2$ and $k_3$.
			It follows that,  	$$E_{k_1}E_{k_2}E_{k_3}=\langle-1,  \varepsilon_{2}, \varepsilon_{p},   \sqrt{\varepsilon_{q}}, \sqrt{\varepsilon_{2q}},   {\varepsilon_{2pq}} ,\sqrt{ \varepsilon_{pq}}, \sqrt{\varepsilon_{2}\varepsilon_{p}\varepsilon_{2p}}\rangle.$$	
			Thus we shall determine elements of $E_{k_1}E_{k_2}E_{k_3}$ which are squares in $\KK$. 
			Notice that by Lemma \ref{lm expressions of units under cond 3 case: 7=1 mod 8 and ()=1},
			$\varepsilon_{2pq}$ is a square in  $\KK$.
			Let  $\xi$ is an element of $\KK$ which is the  square root of an element of $E_{k_1}E_{k_2}E_{k_3}$. We can assume that
			$$\xi^2=\varepsilon_{2}^a\varepsilon_{p}^b \sqrt{\varepsilon_{q}}^c\sqrt{\varepsilon_{2q}}^d\sqrt{\varepsilon_{pq}}^f
			\sqrt{\varepsilon_{2}\varepsilon_{p}\varepsilon_{2p}}^g,$$
			where $a, b, c, d,   f$ and $g$ are in $\{0, 1\}$.

 	\noindent\ding{224}  Let us start	by applying   the norm map $N_{\KK/k_2}=1+\tau_2$. We have 
 $\sqrt{\varepsilon_{2}\varepsilon_{p}\varepsilon_{2p}}^{1+ \tau_2}=(-1)^v\varepsilon_{2}$, for some $v\in\{0,1\}$. By means of 
 \eqref{norm q=7 p=1mod4q} and Table \ref{tab3 case: 7=1 mod 8 and ()=1},   we get:
 \begin{eqnarray*}
 	N_{\KK/k_2}(\xi^2)&=&
 	\varepsilon_{2}^{2a}\cdot(-1)^b \cdot \varepsilon_{q}^c\cdot\varepsilon_{2q}^d\cdot1 \cdot (-1)^{gv}\varepsilon_{2}^{g}\\
 	&=&	\varepsilon_{2}^{2a}  \varepsilon_{q}^c\varepsilon_{2q}^d\cdot(-1)^{b+  gv} \varepsilon_{2}^g.
 \end{eqnarray*}
 
 Thus    $b+  gv= 0\pmod2$ and $g=0$. Therefore,  $b= 0$ and
$$\xi^2=\varepsilon_{2}^a  \sqrt{\varepsilon_{q}}^c\sqrt{\varepsilon_{2q}}^d\sqrt{\varepsilon_{pq}}^f.$$	
 
 \noindent\ding{224} Let us apply the norm $N_{\KK/k_5}=1+\tau_1\tau_2$, with $k_5=\QQ(\sqrt{q}, \sqrt{2p})$. By the table \eqref{norm q=7 p=1mod4q} and Table \ref{tab3 case: 7=1 mod 8 and ()=1}, we  have:	
 \begin{eqnarray*}
 	N_{\KK/k_5}(\xi^2)&=&(-1)^a \cdot (-1)^c\cdot \varepsilon_{  q}^c\cdot(-1)^d \cdot (-1)^f \\
 	&=&	 (-1)^{a+ c+d +f }\varepsilon_{  q}^c   .
 \end{eqnarray*}
 Thus $a+c+d+f  =0\pmod 2$ and $c=0$. Thus, $a+ d+f  =0\pmod 2$. Therefore, 
$$\xi^2=\varepsilon_{2}^a   \sqrt{\varepsilon_{2q}}^d\sqrt{\varepsilon_{pq}}^f.$$

  \noindent\ding{224} Let us apply the norm $N_{\KK/k_6}=1+\tau_1\tau_3$, with $k_6=\QQ(\sqrt{p}, \sqrt{2q})$. By the table \eqref{norm q=7 p=1mod4q} and Table \ref{tab3 case: 7=1 mod 8 and ()=1}, we  have:
 \begin{eqnarray*}
 	N_{\KK/k_6}(\xi^2)&=&(-1)^a  \cdot (-1)^d\cdot \varepsilon_{2  q}^d\cdot  (-1)^f\\
 	&=&	   (-1)^{a+d+f}\varepsilon_{2  q}^d.
 \end{eqnarray*}
 Thus $d=0$ and
 $a=f$. Therefore,
 $$\xi^2=\varepsilon_{2}^a    \sqrt{\varepsilon_{pq}}^a.$$

	\noindent\ding{224} Let us apply the norm $N_{\KK/k_4}=1+\tau_1$, with $k_4=\QQ(\sqrt{p}, \sqrt{q})$.  By the table \eqref{norm q=7 p=1mod4q} and Table \ref{tab3 case: 7=1 mod 8 and ()=1}, we  have:		
\begin{eqnarray*}
	N_{\KK/k_4}(\xi^2)&=&(-1)^a\cdot (-1)^a\cdot \varepsilon_{  pq}^a=\varepsilon_{  pq}^a.
\end{eqnarray*}
Thus   $a=0$.
It follows that the only element of $E_{k_1}E_{k_2}E_{k_3}$ that is a square in $\KK$ is $\varepsilon_{2pq}$ and so we have  the first item.

	\item  Assume that   $N(\varepsilon_{2p})=1$. So we have:
	 $$E_{k_1}E_{k_2}E_{k_3}=\langle-1,  \varepsilon_{2}, \varepsilon_{p},   \sqrt{\varepsilon_{q}}, \sqrt{\varepsilon_{2q}},   {\varepsilon_{2pq}} ,\sqrt{ \varepsilon_{pq}}, \sqrt{\varepsilon_{2p}}\rangle.$$	
	To determine the elements of $E_{k_1}E_{k_2}E_{k_3}$ which are squares in $\KK$, let us consider $\xi$   an element of $\KK$ which is the  square root of an element of $E_{k_1}E_{k_2}E_{k_3}$. As $\varepsilon_{2pq}$ is a square in  $\KK$, we can assume that
	$$\xi^2=\varepsilon_{2}^a\varepsilon_{p}^b \sqrt{\varepsilon_{q}}^c\sqrt{\varepsilon_{2q}}^d\sqrt{\varepsilon_{pq}}^f
	\sqrt{\varepsilon_{2p}}^g,$$
	where $a, b, c, d,   f$ and $g$ are in $\{0, 1\}$.
	
	\noindent\ding{224}  Let us start	by applying   the norm map $N_{\KK/k_2}=1+\tau_2$.  By \eqref{T_3_-1_eqi2p_N=1}, 
	\eqref{norm q=7 p=1mod4q} and Table \ref{tab3 case: 7=1 mod 8 and ()=1},   we have:
	\begin{eqnarray*}
		N_{\KK/k_2}(\xi^2)&=&
		\varepsilon_{2}^{2a}\cdot(-1)^b \cdot \varepsilon_{q}^c\cdot\varepsilon_{2q}^d\cdot1 \cdot (-1)^{gu} \\
		&=&	\varepsilon_{2}^{2a}  \varepsilon_{q}^c\varepsilon_{2q}^d\cdot(-1)^{b+ gu}  .
	\end{eqnarray*}
	
	Thus    $b+  gu= 0\pmod2$. 
	
	 \noindent\ding{224} Let us apply the norm $N_{\KK/k_5}=1+\tau_1\tau_2$, with $k_5=\QQ(\sqrt{q}, \sqrt{2p})$. By the table \eqref{norm q=7 p=1mod4q} and Table \ref{tab3 case: 7=1 mod 8 and ()=1}, we  have:	
	\begin{eqnarray*}
		N_{\KK/k_5}(\xi^2)&=&(-1)^a \cdot (-1)^b \cdot(-1)^c\cdot \varepsilon_{  q}^c\cdot(-1)^d \cdot (-1)^f \cdot (-1)^g\cdot \varepsilon_{  2p}^g\\
		&=&	 (-1)^{a+b+ c+d +f+g }\varepsilon_{  q}^c\cdot \varepsilon_{  2p}^g  .
	\end{eqnarray*}
	Thus $c=g$ and so $a+b+  d +f   =0\pmod 2$. Thus, $a+ d+f  =0\pmod 2$. Therefore, 
$$\xi^2=\varepsilon_{2}^a\varepsilon_{p}^b \sqrt{\varepsilon_{q}}^c\sqrt{\varepsilon_{2q}}^d\sqrt{\varepsilon_{pq}}^f
\sqrt{\varepsilon_{2p}}^c.$$
	
\noindent\ding{224} Let us apply the norm $N_{\KK/k_6}=1+\tau_1\tau_3$, with $k_6=\QQ(\sqrt{p}, \sqrt{2q})$. By the table \eqref{norm q=7 p=1mod4q} and Table \ref{tab3 case: 7=1 mod 8 and ()=1}, we  have:
\begin{eqnarray*}
	N_{\KK/k_6}(\xi^2)&=&(-1)^a\cdot \varepsilon_{p}^{2b}  \cdot(-1)^c  \cdot (-1)^d\cdot \varepsilon_{2  q}^d\cdot  (-1)^f\cdot  (-1)^{cu+c}\\
	&=&	  \varepsilon_{p}^{2b} (-1)^{a+d+f+cu}\varepsilon_{2  q}^d.
\end{eqnarray*}	
	Thus  $d=0$ and so  $a+ f+cu= 0\pmod2$.	Therefore,
	$$\xi^2=\varepsilon_{2}^a\varepsilon_{p}^b \sqrt{\varepsilon_{q}}^c \sqrt{\varepsilon_{pq}}^f
	\sqrt{\varepsilon_{2p}}^c.$$
	
	\noindent\ding{224} Let us apply the norm $N_{\KK/k_4}=1+\tau_1$, with $k_4=\QQ(\sqrt{p}, \sqrt{q})$.  By the table \eqref{norm q=7 p=1mod4q} and Table \ref{tab3 case: 7=1 mod 8 and ()=1}, we  have:		
\begin{eqnarray*}
	N_{\KK/k_4}(\xi^2)&=&(-1)^a\cdot \varepsilon_{p}^{2b}\cdot (-1)^c\cdot \varepsilon_{q}^{c}\cdot(-1)^f\cdot \varepsilon_{  pq}^f\cdot(-1)^{cu+c}, \\
	&=&\varepsilon_{p}^{2b}(-1)^{a+f+cu}\varepsilon_{q}^{c}\varepsilon_{  pq}^f.
\end{eqnarray*}
Thus   $c=f$. Since  $b+  gu= 0=b+cu\pmod2$, we have	 
	$$\xi^2=\varepsilon_{2}^a\varepsilon_{p}^{cu} \sqrt{\varepsilon_{q}}^c \sqrt{\varepsilon_{pq}}^c
	\sqrt{\varepsilon_{2p}}^c,$$
with $a+ c+cu= 0\pmod2$.		
\end{enumerate}
\end{proof}

	 \begin{exam}Using  PARI/GP calculator version  2.15.5 (64bit), Feb 11 2024, we  get the following examples which illustrate the above theorem.
		\begin{enumerate}[\rm $1)$]
			\item Consider the prime numbers $p=313$ and  $q=463$. In this case, the   conditions of the first item of the above theorem are verified and we have:
		$$E_{\KK}=\langle -1,   \varepsilon_{2}, \varepsilon_{p} ,    \sqrt{\varepsilon_{q}}, \sqrt{\varepsilon_{2q}},\sqrt{ \varepsilon_{pq}} ,\sqrt{ \varepsilon_{2pq}}, \sqrt{\varepsilon_{2}\varepsilon_{p}\varepsilon_{2p}} \rangle.$$
			
			Furthermore, we have $h_2(2p)=4$, $h_2(pq)=8$,  $h_2(2pq)=4$ and  $h_2(\KK)=8$.
			\item Consider the prime numbers $p=17$ and  $q=223$. In this case, the   conditions of the second item of the above theorem are verified and we have:
			$$E_{\KK}=\langle -1,   \varepsilon_{2}, \varepsilon_{p},   \sqrt{\varepsilon_{q}}, \sqrt{\varepsilon_{2q}},  \sqrt{\varepsilon_{pq}} ,\sqrt{ \varepsilon_{2pq}},    \sqrt{\varepsilon_{2}^{a}\varepsilon_{p}^{u}     \sqrt{\varepsilon_{ q}}\sqrt{\varepsilon_{ pq}}
				\sqrt{\varepsilon_{2p}} }  \rangle$$
			for some 	$a\in\{0,1\}$ such that $a \equiv 1+u\pmod2$.	Furthermore, we have $h_2(2p)=2$, $h_2(pq)=4$,  $h_2(2pq)=8$ and  $h_2(\KK)=8$.
		\end{enumerate}
	\end{exam}

\begin{theorem}\label{T_7_1-C5} Let $p\equiv 1\pmod{8}$ and $q\equiv7\pmod 8$ be two primes such that     $\genfrac(){}{0}{p}{q} =1$. 
	Put     $\KK=\QQ(\sqrt 2, \sqrt{p}, \sqrt{q} )$.  Assume furthermore that $p(x+1)$ and $p(v+1)$ are squares in $\mathbb{N}$, where $x$ and $v$ are defined in Lemma \ref{lm expressions of units under cond 3 case: 7=1 mod 8 and ()=1}.
	\begin{enumerate}[\rm 1)]
		\item Assume that $N(\varepsilon_{2p})=-1$.  We have
		\begin{enumerate}[\rm $\bullet$]
			
			\item The unit group of $\KK$ is :
			$$E_{\KK}=\langle -1,   \varepsilon_{2}, \varepsilon_{p} ,    \sqrt{\varepsilon_{q}}, \sqrt{\varepsilon_{2q}},\sqrt{ \varepsilon_{pq}} ,\sqrt{ \varepsilon_{2pq}}, \sqrt{\varepsilon_{2}\varepsilon_{p}\varepsilon_{2p}} \rangle.$$
			\item  The $2$-class number  of  $\KK$ equals $ \frac{1}{2^{4}} h_2(2p)h_2(pq)h_2(2pq)$.  
		\end{enumerate}
		\item Assume that $N(\varepsilon_{2p})=1$ and let $a\in\{0,1\}$ such that $a \equiv 1+u\pmod2$.   We have
		
		\begin{enumerate}[\rm $\bullet$]
			\item  The unit group of $\KK$ is :
			$$E_{\KK}=\langle -1,   \varepsilon_{2}, \varepsilon_{p},   \sqrt{\varepsilon_{q}}, \sqrt{\varepsilon_{2q}},  \sqrt{\varepsilon_{pq}} ,\sqrt{ \varepsilon_{2pq}},    \sqrt{    \varepsilon_{2}^{a\alpha}\varepsilon_{p}^{u\alpha} \sqrt{\varepsilon_{q}}^\alpha\sqrt{\varepsilon_{2q}}^\alpha\sqrt{\varepsilon_{pq}\varepsilon_{2pq}}^\alpha
				\sqrt{\varepsilon_{2p}}^{1+\gamma}  }  \rangle$$
			where $\alpha$, $\gamma\in \{0,1\}$ such that $\alpha\not=\gamma$ and $\alpha =1$ if and only if $\varepsilon_{2}^{a}\varepsilon_{p}^{u} \sqrt{\varepsilon_{q}}\sqrt{\varepsilon_{2q}}\sqrt{\varepsilon_{pq}\varepsilon_{2pq}}
			\sqrt{\varepsilon_{2p}}$ is a square in $\KK$.
			\item  The $2$-class number  of  $\KK$ equals $ \frac{1}{2^{4-\alpha}} h_2(2p)h_2(pq)h_2(2pq)$.
		\end{enumerate}
	\end{enumerate}
\end{theorem}
\begin{proof}

 	\begin{enumerate}[\rm 1)]
		\item 		Assume that   $N(\varepsilon_{2p})=-1$.  By Lemma \ref{units of k1},     $\{\varepsilon_{2}, \varepsilon_{p},	\sqrt{\varepsilon_{2}\varepsilon_{p}\varepsilon_{2p}}\}$  is a  fundamental system of units of $k_1$. Using Lemmas \ref{lm expressions of units q_2q} and \ref{lm expressions of units under cond 3 case: 7=1 mod 8 and ()=1}, we check   that
		 $\{\varepsilon_{2}, \sqrt{\varepsilon_{q}}, \sqrt{\varepsilon_{2q}}\}$ and $\{ \varepsilon_{2}, 	{\varepsilon_{ pq}},\sqrt{\varepsilon_{pq}\varepsilon_{2pq}}\}$ are respectively fundamental systems of units of $k_2$ and $k_3$.
		It follows that,  	$$E_{k_1}E_{k_2}E_{k_3}=\langle-1,  \varepsilon_{2}, \varepsilon_{p},   \sqrt{\varepsilon_{q}}, \sqrt{\varepsilon_{2q}},   {\varepsilon_{ pq}} ,\sqrt{ \varepsilon_{pq}\varepsilon_{2pq}}, \sqrt{\varepsilon_{2}\varepsilon_{p}\varepsilon_{2p}}\rangle.$$	
		Thus we shall determine elements of $E_{k_1}E_{k_2}E_{k_3}$ which are squares in $\KK$. 
		Notice that by Lemma \ref{lm expressions of units under cond 3 case: 7=1 mod 8 and ()=1},
		$\varepsilon_{pq}$ is a square in  $\KK$.
		Let  $\xi$ is an element of $\KK$ which is the  square root of an element of $E_{k_1}E_{k_2}E_{k_3}$. We can assume that
		$$\xi^2=\varepsilon_{2}^a\varepsilon_{p}^b \sqrt{\varepsilon_{q}}^c\sqrt{\varepsilon_{2q}}^d\sqrt{\varepsilon_{pq}\varepsilon_{2pq}}^f 
		\sqrt{\varepsilon_{2}\varepsilon_{p}\varepsilon_{2p}}^g,$$
		where $a, b, c, d,   f$ and $g$ are in $\{0, 1\}$.

		\noindent\ding{224}  Let us start	by applying   the norm map $N_{\KK/k_2}=1+\tau_2$. We have 
		$\sqrt{\varepsilon_{2}\varepsilon_{p}\varepsilon_{2p}}^{1+ \tau_2}=(-1)^v\varepsilon_{2}$, for some $v\in\{0,1\}$. By means of 
		\eqref{norm q=7 p=1mod4q} and Table \ref{tab3 case: 7=1 mod 8 and ()=1},   we get:
		\begin{eqnarray*}
			N_{\KK/k_2}(\xi^2)&=&
			\varepsilon_{2}^{2a}\cdot(-1)^b \cdot \varepsilon_{q}^c\cdot\varepsilon_{2q}^d\cdot1 \cdot (-1)^{gv}\varepsilon_{2}^{g}\\
			&=&	\varepsilon_{2}^{2a}  \varepsilon_{q}^c\varepsilon_{2q}^d\cdot(-1)^{b+  gv} \varepsilon_{2}^g.
		\end{eqnarray*}
		
		Thus    $b+  gv= 0\pmod2$ and $g=0$. Therefore,  $b= 0$ and
	$$\xi^2=\varepsilon_{2}^a  \sqrt{\varepsilon_{q}}^c\sqrt{\varepsilon_{2q}}^d\sqrt{\varepsilon_{pq}\varepsilon_{2pq}}^f 
 .$$
 
  \noindent\ding{224} Let us apply the norm $N_{\KK/k_5}=1+\tau_1\tau_2$, with $k_5=\QQ(\sqrt{q}, \sqrt{2p})$. By the table \eqref{norm q=7 p=1mod4q} and Table \ref{tab3 case: 7=1 mod 8 and ()=1}, we  have:	
 \begin{eqnarray*}
 	N_{\KK/k_5}(\xi^2)&=&(-1)^a   \cdot(-1)^c\cdot \varepsilon_{  q}^c\cdot(-1)^d \cdot\varepsilon_{2pq}^f  \\
 	&=&	\varepsilon_{2pq}^f (-1)^{a+ c+d  }\varepsilon_{  q}^c   .
 \end{eqnarray*}
 Thus $c=0$ and so $a =d$.   Therefore, 
 $$\xi^2=\varepsilon_{2}^a   \sqrt{\varepsilon_{2q}}^a\sqrt{\varepsilon_{pq}\varepsilon_{2pq}}^f 
 .$$
 
	\noindent\ding{224} Let us apply the norm $N_{\KK/k_6}=1+\tau_1\tau_3$, with $k_6=\QQ(\sqrt{p}, \sqrt{2q})$. By the table \eqref{norm q=7 p=1mod4q} and Table \ref{tab3 case: 7=1 mod 8 and ()=1}, we  have:
	\begin{eqnarray*}
		N_{\KK/k_6}(\xi^2)&=&(-1)^a  \cdot (-1)^a\cdot \varepsilon_{2  q}^a\cdot  \varepsilon_{2pq}^f\\
		&=&	    \varepsilon_{2  q}^a\varepsilon_{2pq}^f.
	\end{eqnarray*}	
	Thus  $a =f$.	Therefore,
	 $$\xi^2=\varepsilon_{2}^a   \sqrt{\varepsilon_{2q}}^a\sqrt{\varepsilon_{pq}\varepsilon_{2pq}}^a .$$

	\noindent\ding{224} Let us apply the norm $N_{\KK/k_4}=1+\tau_1$, with $k_4=\QQ(\sqrt{p}, \sqrt{q})$.  By the table \eqref{norm q=7 p=1mod4q} and Table \ref{tab3 case: 7=1 mod 8 and ()=1}, we  have:		
\begin{eqnarray*}
	N_{\KK/k_4}(\xi^2)&=&(-1)^a\cdot (-1)^a\cdot \varepsilon_{  pq}^a=\varepsilon_{  pq}^a.
\end{eqnarray*}
Thus   $a=0$.
It follows that the only element of $E_{k_1}E_{k_2}E_{k_3}$ that is a square in $\KK$ is $\varepsilon_{pq}$ and so we have  the first item.

	\item  Assume that   $N(\varepsilon_{2p})=1$. So we have:
	$$E_{k_1}E_{k_2}E_{k_3}=\langle-1,  \varepsilon_{2}, \varepsilon_{p},   \sqrt{\varepsilon_{q}}, \sqrt{\varepsilon_{2q}},   {\varepsilon_{ pq}} ,\sqrt{ \varepsilon_{pq}\varepsilon_{2pq}}, \sqrt{\varepsilon_{2p}}\rangle.$$	
	To determine the elements of $E_{k_1}E_{k_2}E_{k_3}$ which are squares in $\KK$, let us consider $\xi$   an element of $\KK$ which is the  square root of an element of $E_{k_1}E_{k_2}E_{k_3}$. As $\varepsilon_{pq}$ is a square in  $\KK$, we can assume that
	$$\xi^2=\varepsilon_{2}^a\varepsilon_{p}^b \sqrt{\varepsilon_{q}}^c\sqrt{\varepsilon_{2q}}^d\sqrt{\varepsilon_{pq}\varepsilon_{2pq}}^f
	\sqrt{\varepsilon_{2p}}^g,$$
	where $a, b, c, d,   f$ and $g$ are in $\{0, 1\}$.
	
	\noindent\ding{224}  Let us start	by applying   the norm map $N_{\KK/k_2}=1+\tau_2$.  By \eqref{T_3_-1_eqi2p_N=1}, 
	\eqref{norm q=7 p=1mod4q} and Table \ref{tab3 case: 7=1 mod 8 and ()=1},   we have:
	\begin{eqnarray*}
		N_{\KK/k_2}(\xi^2)&=&
		\varepsilon_{2}^{2a}\cdot(-1)^b \cdot \varepsilon_{q}^c\cdot\varepsilon_{2q}^d\cdot1 \cdot (-1)^{gu} \\
		&=&	\varepsilon_{2}^{2a}  \varepsilon_{q}^c\varepsilon_{2q}^d\cdot(-1)^{b+ gu}  .
	\end{eqnarray*}
	
	Thus    $b+  gu= 0\pmod2$. So $b=gu$.
	
 \noindent\ding{224} Let us apply the norm $N_{\KK/k_5}=1+\tau_1\tau_2$, with $k_5=\QQ(\sqrt{q}, \sqrt{2p})$. By the table \eqref{norm q=7 p=1mod4q} and Table \ref{tab3 case: 7=1 mod 8 and ()=1}, we  have:	
\begin{eqnarray*}
	N_{\KK/k_5}(\xi^2)&=&(-1)^a \cdot (-1)^b \cdot(-1)^c\cdot \varepsilon_{  q}^c\cdot(-1)^d \cdot \varepsilon_{2pq}^f\cdot (-1)^g \cdot\varepsilon_{2p}^g\\
	&=&	\varepsilon_{2pq}^f (-1)^{a+b+ c+d+g  }\varepsilon_{  q}^c \varepsilon_{2p}^g  .
\end{eqnarray*}
Thus $c=g$ and so $a+b +d     =0\pmod 2$.   Therefore, 
$$\xi^2=\varepsilon_{2}^a\varepsilon_{p}^b \sqrt{\varepsilon_{q}}^c\sqrt{\varepsilon_{2q}}^d\sqrt{\varepsilon_{pq}\varepsilon_{2pq}}^f
\sqrt{\varepsilon_{2p}}^c.$$
	
	\noindent\ding{224} Let us apply the norm $N_{\KK/k_6}=1+\tau_1\tau_3$, with $k_6=\QQ(\sqrt{p}, \sqrt{2q})$. By the table \eqref{norm q=7 p=1mod4q} and Table \ref{tab3 case: 7=1 mod 8 and ()=1}, we  have:
\begin{eqnarray*}
	N_{\KK/k_6}(\xi^2)&=&(-1)^a\cdot \varepsilon_{p}^{2b} \cdot (-1)^c  \cdot (-1)^d\cdot \varepsilon_{2  q}^d\cdot  \varepsilon_{2pq}^f\cdot (-1)^{cu+c}\\
	&=&	   \varepsilon_{p}^{2b} (-1)^{a+d+cu}\varepsilon_{2  q}^d\varepsilon_{2pq}^f.
\end{eqnarray*}	
Thus  $a+d+cu=0\pmod 2$ and $d=f$. Therefore,
$$\xi^2=\varepsilon_{2}^a\varepsilon_{p}^b \sqrt{\varepsilon_{q}}^c\sqrt{\varepsilon_{2q}}^d\sqrt{\varepsilon_{pq}\varepsilon_{2pq}}^d
\sqrt{\varepsilon_{2p}}^c.$$
	
		\noindent\ding{224} Let us apply the norm $N_{\KK/k_4}=1+\tau_1$, with $k_4=\QQ(\sqrt{p}, \sqrt{q})$.  By the table \eqref{norm q=7 p=1mod4q} and Table \ref{tab3 case: 7=1 mod 8 and ()=1}, we  have:		
	\begin{eqnarray*}
		N_{\KK/k_4}(\xi^2)&=&(-1)^a\cdot \varepsilon_{p}^{2b}\cdot (-1)^c\cdot \varepsilon_{q}^{c} \cdot (-1)^d\cdot \varepsilon_{  pq}^d\cdot (-1)^{cu+c}\\
		&=& \varepsilon_{p}^{2b} (-1)^{a+d+cu}\varepsilon_{q}^{c}\varepsilon_{  pq}^d.
	\end{eqnarray*}
	Thus   $d=c$ and so  $a+c+cu=0\pmod{2}$. It follows that 
 
$$\xi^2=\varepsilon_{2}^a\varepsilon_{p}^{cu} \sqrt{\varepsilon_{q}}^c\sqrt{\varepsilon_{2q}}^c\sqrt{\varepsilon_{pq}\varepsilon_{2pq}}^c
\sqrt{\varepsilon_{2p}}^c,$$	
where 	$a+cu+c=0\pmod 2$.

		\end{enumerate}

\end{proof}

 \begin{exam}Using  PARI/GP calculator version  2.15.5 (64bit), Feb 11 2024, we  get the following examples which illustrate the above theorem.
	\begin{enumerate}[\rm $1)$]
		\item  Consider the prime numbers $p=41$ and  $q=223$. In this case, the   conditions of the first item of the above theorem are verified and we have:
		$$E_{\KK}=\langle -1,   \varepsilon_{2}, \varepsilon_{p} ,    \sqrt{\varepsilon_{q}}, \sqrt{\varepsilon_{2q}},\sqrt{ \varepsilon_{pq}} ,\sqrt{ \varepsilon_{2pq}}, \sqrt{\varepsilon_{2}\varepsilon_{p}\varepsilon_{2p}} \rangle.$$
		
		Furthermore, we have $h_2(2p)=4$, $h_2(pq)=8$,  $h_2(2pq)=4$ and  $h_2(\KK)=8$.
		\item Consider the prime numbers $p=257$ and  $q=79$. In this case, the   conditions of the second item of the above theorem are verified and we have:
			$$E_{\KK}=\langle -1,   \varepsilon_{2}, \varepsilon_{p},   \sqrt{\varepsilon_{q}}, \sqrt{\varepsilon_{2q}},  \sqrt{\varepsilon_{pq}} ,\sqrt{ \varepsilon_{2pq}},    \sqrt{    \varepsilon_{2}^{a}\varepsilon_{p}^{u} \sqrt{\varepsilon_{q}}\sqrt{\varepsilon_{2q}}\sqrt{\varepsilon_{pq}\varepsilon_{2pq}}
			\sqrt{\varepsilon_{2p}}  }  \rangle$$
		for some 	$a\in\{0,1\}$ such that $a \equiv 1+u\pmod2$.	Furthermore, we have $h_2(2p)=4$, $h_2(pq)=4$,  $h_2(2pq)=4$ and  $h_2(\KK)=8$.
	\end{enumerate}
\end{exam}

\begin{theorem}\label{T_7_1-C6} Let $p\equiv 1\pmod{8}$ and $q\equiv7\pmod 8$ be two primes such that     $\genfrac(){}{0}{p}{q} =1$. 
	Put     $\KK=\QQ(\sqrt 2, \sqrt{p}, \sqrt{q} )$.  Assume furthermore that $p(x+1)$ and $2p(v+1)$ are squares in $\mathbb{N}$, where $x$ and $v$ are defined in Lemma \ref{lm expressions of units under cond 3 case: 7=1 mod 8 and ()=1}.
	\begin{enumerate}[\rm 1)]
		\item Assume that $N(\varepsilon_{2p})=-1$.  We have
		\begin{enumerate}[\rm $\bullet$]
			
			\item The unit group of $\KK$ is :
			$$E_{\KK}=\langle -1,   \varepsilon_{2}, \varepsilon_{p} ,    \sqrt{\varepsilon_{q}},  \sqrt{ \varepsilon_{pq}} ,\sqrt{ \varepsilon_{2pq}}, \sqrt{\varepsilon_{2}\varepsilon_{p}\varepsilon_{2p}}, \sqrt{  \sqrt{\varepsilon_{2q}}^{1+\gamma}\sqrt{\varepsilon_{pq}\varepsilon_{2pq}}^\alpha    } \rangle,$$
			where $\alpha$, $\gamma\in \{0,1\}$ such that $\alpha\not=\gamma$ and $\alpha =1$ if and only if $\sqrt{\varepsilon_{2q}}\sqrt{\varepsilon_{pq}\varepsilon_{2pq}}$ is a square in $\KK$.
			
			\item  The $2$-class number  of  $\KK$ equals $ \frac{1}{2^{4-\alpha}} h_2(2p)h_2(pq)h_2(2pq)$.  
		\end{enumerate}
		\item Assume that $N(\varepsilon_{2p})=1$.   We have
		
		\begin{enumerate}[\rm $\bullet$]
			\item  The unit group of $\KK$ is :
			$$E_{\KK}=\langle -1,   \varepsilon_{2}, \varepsilon_{p},   \sqrt{\varepsilon_{q}},   \sqrt{\varepsilon_{pq}} ,\sqrt{ \varepsilon_{2pq}},  \sqrt{\varepsilon_{2p}},  \sqrt{  \sqrt{\varepsilon_{2q}}^{1+\gamma}\sqrt{\varepsilon_{pq}\varepsilon_{2pq}}^\alpha    }  \rangle$$
			where $\alpha$, $\gamma\in \{0,1\}$ such that $\alpha\not=\gamma$ and $\alpha =1$ if and only if $\sqrt{\varepsilon_{2q}}\sqrt{\varepsilon_{pq}\varepsilon_{2pq}}$ is a square in $\KK$.
			\item  The $2$-class number  of  $\KK$ equals $ \frac{1}{2^{4-\alpha}} h_2(2p)h_2(pq)h_2(2pq)$.
		\end{enumerate}
	\end{enumerate}
\end{theorem}
\begin{proof} 
 	\begin{enumerate}[\rm 1)]
		\item 		Assume that   $N(\varepsilon_{2p})=-1$.  By Lemma \ref{units of k1},     $\{\varepsilon_{2}, \varepsilon_{p},	\sqrt{\varepsilon_{2}\varepsilon_{p}\varepsilon_{2p}}\}$  is a  fundamental system of units of $k_1$. Using Lemmas \ref{lm expressions of units q_2q} and \ref{lm expressions of units under cond 3 case: 7=1 mod 8 and ()=1}, we check   that
		$\{\varepsilon_{2}, \sqrt{\varepsilon_{q}}, \sqrt{\varepsilon_{2q}}\}$ and $\{ \varepsilon_{2}, 	{\varepsilon_{ pq}},\sqrt{\varepsilon_{pq}\varepsilon_{2pq}}\}$ are respectively fundamental systems of units of $k_2$ and $k_3$.
		It follows that,  	$$E_{k_1}E_{k_2}E_{k_3}=\langle-1,  \varepsilon_{2}, \varepsilon_{p},   \sqrt{\varepsilon_{q}}, \sqrt{\varepsilon_{2q}},   {\varepsilon_{ pq}} ,\sqrt{ \varepsilon_{pq}\varepsilon_{2pq}}, \sqrt{\varepsilon_{2}\varepsilon_{p}\varepsilon_{2p}}\rangle.$$	
		Thus we shall determine elements of $E_{k_1}E_{k_2}E_{k_3}$ which are squares in $\KK$. 
		Notice that by Lemma \ref{lm expressions of units under cond 3 case: 7=1 mod 8 and ()=1},
		$\varepsilon_{pq}$ is a square in  $\KK$.
		Let  $\xi$ is an element of $\KK$ which is the  square root of an element of $E_{k_1}E_{k_2}E_{k_3}$. We can assume that
		$$\xi^2=\varepsilon_{2}^a\varepsilon_{p}^b \sqrt{\varepsilon_{q}}^c\sqrt{\varepsilon_{2q}}^d\sqrt{\varepsilon_{pq}\varepsilon_{2pq}}^f 
		\sqrt{\varepsilon_{2}\varepsilon_{p}\varepsilon_{2p}}^g,$$
		where $a, b, c, d,   f$ and $g$ are in $\{0, 1\}$.

		\noindent\ding{224}  Let us start	by applying   the norm map $N_{\KK/k_2}=1+\tau_2$. We have 
		$\sqrt{\varepsilon_{2}\varepsilon_{p}\varepsilon_{2p}}^{1+ \tau_2}=(-1)^v\varepsilon_{2}$, for some $v\in\{0,1\}$. By means of 
		\eqref{norm q=7 p=1mod4q} and Table \ref{tab3 case: 7=1 mod 8 and ()=1},   we get:
		\begin{eqnarray*}
			N_{\KK/k_2}(\xi^2)&=&
			\varepsilon_{2}^{2a}\cdot(-1)^b \cdot \varepsilon_{q}^c\cdot\varepsilon_{2q}^d\cdot1 \cdot (-1)^{gv}\varepsilon_{2}^{g}\\
			&=&	\varepsilon_{2}^{2a}  \varepsilon_{q}^c\varepsilon_{2q}^d\cdot(-1)^{b+  gv} \varepsilon_{2}^g.
		\end{eqnarray*}
		
		Thus    $b+  gv= 0\pmod2$ and $g=0$. Therefore,  $b= 0$ and
		$$\xi^2=\varepsilon_{2}^a  \sqrt{\varepsilon_{q}}^c\sqrt{\varepsilon_{2q}}^d\sqrt{\varepsilon_{pq}\varepsilon_{2pq}}^f 
		.$$

		\noindent\ding{224} Let us apply the norm $N_{\KK/k_5}=1+\tau_1\tau_2$, with $k_5=\QQ(\sqrt{q}, \sqrt{2p})$. By the table \eqref{norm q=7 p=1mod4q} and Table \ref{tab3 case: 7=1 mod 8 and ()=1}, we  have:	
		\begin{eqnarray*}
			N_{\KK/k_5}(\xi^2)&=&(-1)^a  \cdot(-1)^c\cdot \varepsilon_{  q}^c\cdot(-1)^d \cdot(-1)^f \cdot\varepsilon_{2pq}^f  \\
			&=&	\varepsilon_{2pq}^f (-1)^{a + c+d +f }\varepsilon_{  q}^c   .
		\end{eqnarray*}
	By Lemma \ref{lm expressions of units under cond 3 case: 7=1 mod 8 and ()=1} (resp. Lemma \ref{lm expressions of units q_2q}), $\varepsilon_{2pq}$ is (resp. ${\varepsilon_{  q}}$ is not)   a square in $k_5$, then $c=0$ and so $a  +d+f =0\pmod 2$.   Therefore, 
		$$\xi^2=\varepsilon_{2}^a   \sqrt{\varepsilon_{2q}}^d\sqrt{\varepsilon_{pq}\varepsilon_{2pq}}^f 
		.$$
		
		\noindent\ding{224} Let us apply the norm $N_{\KK/k_6}=1+\tau_1\tau_3$, with $k_6=\QQ(\sqrt{p}, \sqrt{2q})$. By the table \eqref{norm q=7 p=1mod4q} and Table \ref{tab3 case: 7=1 mod 8 and ()=1}, we  have:
	\begin{eqnarray*}
		N_{\KK/k_6}(\xi^2)&=&(-1)^a  \cdot (-1)^d\cdot \varepsilon_{2  q}^d\cdot  (-1)^f \cdot \varepsilon_{2pq}^f\\
		&=&	    (-1)^{a+d+f}\varepsilon_{2  q}^d\varepsilon_{2pq}^f.
	\end{eqnarray*}	
	Thus  $d =f$ and $a=0$.	Therefore,
$$\xi^2=    \sqrt{\varepsilon_{2q}}^d\sqrt{\varepsilon_{pq}\varepsilon_{2pq}}^f .$$
		
	\noindent\ding{224} Let us apply the norm $N_{\KK/k_4}=1+\tau_1$, with $k_4=\QQ(\sqrt{p}, \sqrt{q})$.  By the table \eqref{norm q=7 p=1mod4q} and Table \ref{tab3 case: 7=1 mod 8 and ()=1}, we  have:		
	\begin{eqnarray*}
		N_{\KK/k_4}(\xi^2)&=&(-1)^d\cdot (-1)^f\cdot \varepsilon_{  pq}^f\\
		&=&\varepsilon_{  pq}^f(-1)^{d+f}.
	\end{eqnarray*}
	Thus   $d=f$.	Therefore,
	 	$$\xi^2=    \sqrt{\varepsilon_{2q}}^d\sqrt{\varepsilon_{pq}\varepsilon_{2pq}}^d .$$
		\item  
		Assume that   $N(\varepsilon_{2p})=1$. So we have:
		$$E_{k_1}E_{k_2}E_{k_3}=\langle-1,  \varepsilon_{2}, \varepsilon_{p},   \sqrt{\varepsilon_{q}}, \sqrt{\varepsilon_{2q}},   {\varepsilon_{ pq}} ,\sqrt{ \varepsilon_{pq}\varepsilon_{2pq}}, \sqrt{\varepsilon_{2p}}\rangle.$$	
		To determine the elements of $E_{k_1}E_{k_2}E_{k_3}$ which are squares in $\KK$, let us consider $\xi$   an element of $\KK$ which is the  square root of an element of $E_{k_1}E_{k_2}E_{k_3}$. As $\varepsilon_{pq}$ is a square in  $\KK$, we can assume that
		$$\xi^2=\varepsilon_{2}^a\varepsilon_{p}^b \sqrt{\varepsilon_{q}}^c\sqrt{\varepsilon_{2q}}^d\sqrt{\varepsilon_{pq}\varepsilon_{2pq}}^f
		\sqrt{\varepsilon_{2p}}^g,$$
		where $a, b, c, d,   f$ and $g$ are in $\{0, 1\}$.
		
		\noindent\ding{224}  Let us start	by applying   the norm map $N_{\KK/k_2}=1+\tau_2$.  By \eqref{T_3_-1_eqi2p_N=1}, 
		\eqref{norm q=7 p=1mod4q} and Table \ref{tab3 case: 7=1 mod 8 and ()=1},   we have:
		\begin{eqnarray*}
			N_{\KK/k_2}(\xi^2)&=&
			\varepsilon_{2}^{2a}\cdot(-1)^b \cdot \varepsilon_{q}^c\cdot\varepsilon_{2q}^d\cdot1 \cdot (-1)^{gu} \\
			&=&	\varepsilon_{2}^{2a}  \varepsilon_{q}^c\varepsilon_{2q}^d\cdot(-1)^{b+ gu}  .
		\end{eqnarray*}
		
		Thus    $b+  gu= 0\pmod2$. So $b=gu$.
		
	\noindent\ding{224} Let us apply the norm $N_{\KK/k_5}=1+\tau_1\tau_2$, with $k_5=\QQ(\sqrt{q}, \sqrt{2p})$. By the table \eqref{norm q=7 p=1mod4q} and Table \ref{tab3 case: 7=1 mod 8 and ()=1}, we  have:	
	\begin{eqnarray*}
		N_{\KK/k_5}(\xi^2)&=&(-1)^a \cdot (-1)^b \cdot(-1)^c\cdot \varepsilon_{  q}^c\cdot(-1)^d\cdot(-1)^f \cdot \varepsilon_{2pq}^f\cdot (-1)^g \cdot\varepsilon_{2p}^g\\
		&=&	\varepsilon_{2pq}^f (-1)^{a+b+ c+d+f+g  }\varepsilon_{  q}^c \varepsilon_{2p}^g  .
	\end{eqnarray*}
	Thus $c=g$ and so $a+b+  d+f      =0\pmod 2$.   Therefore, 
		$$\xi^2=\varepsilon_{2}^a\varepsilon_{p}^b \sqrt{\varepsilon_{q}}^c\sqrt{\varepsilon_{2q}}^d\sqrt{\varepsilon_{pq}\varepsilon_{2pq}}^f
	\sqrt{\varepsilon_{2p}}^c.$$
		
	\noindent\ding{224} Let us apply the norm $N_{\KK/k_6}=1+\tau_1\tau_3$, with $k_6=\QQ(\sqrt{p}, \sqrt{2q})$. By the table  \eqref{norm q=7 p=1mod4q} and Table \ref{tab3 case: 7=1 mod 8 and ()=1}, we  have:
	\begin{eqnarray*}
		N_{\KK/k_6}(\xi^2)&=&(-1)^a\cdot \varepsilon_{p}^{2b} \cdot (-1)^c  \cdot (-1)^d\cdot \varepsilon_{2  q}^d \cdot (-1)^f\cdot  \varepsilon_{2pq}^f\cdot (-1)^{cu+c}\\
		&=&	   \varepsilon_{p}^{2b} (-1)^{a+d+f+cu}\varepsilon_{2  q}^d\varepsilon_{2pq}^f.
	\end{eqnarray*}	
	Thus $d=f$ and so $a  +cu=0\pmod 2$. As $a+b+  d+f      =0\pmod 2$, then $a=b=cu$.
	Therefore,
	$$\xi^2=\varepsilon_{2}^{cu}\varepsilon_{p}^{cu} \sqrt{\varepsilon_{q}}^c\sqrt{\varepsilon_{2q}}^d\sqrt{\varepsilon_{pq}\varepsilon_{2pq}}^d
	\sqrt{\varepsilon_{2p}}^c.$$	
\noindent\ding{224} Let us apply the norm $N_{\KK/k_4}=1+\tau_1$, with $k_4=\QQ(\sqrt{p}, \sqrt{q})$.  By the table \eqref{norm q=7 p=1mod4q} and Table \ref{tab3 case: 7=1 mod 8 and ()=1}, we  have:		
\begin{eqnarray*}
	N_{\KK/k_4}(\xi^2)&=&(-1)^{cu}\cdot \varepsilon_{p}^{2cu}\cdot (-1)^c\cdot \varepsilon_{q}^{c} \cdot (-1)^d\cdot (-1)^d\cdot \varepsilon_{  pq}^d\cdot (-1)^{cu+c}\\
	&=& \varepsilon_{p}^{2cu}\varepsilon_{  pq}^d\cdot \varepsilon_{q}^{c}.
\end{eqnarray*}
Thus   $c=0$. It follows that 
$$\xi^2= \sqrt{\varepsilon_{2q}}^d\sqrt{\varepsilon_{pq}\varepsilon_{2pq}}^d,$$	
\end{enumerate} 
\end{proof} 
	
	 \begin{exam}Using  PARI/GP calculator version  2.15.5 (64bit), Feb 11 2024, we  get the following examples which illustrate the above theorem.
		\begin{enumerate}[\rm $1)$]
			\item  Consider the prime numbers $p=457$ and  $q=463$. In this case, the   conditions of the first item of the above theorem are verified and we have:
			$$E_{\KK}=\langle -1,   \varepsilon_{2}, \varepsilon_{p} ,    \sqrt{\varepsilon_{q}},  \sqrt{ \varepsilon_{pq}} ,\sqrt{ \varepsilon_{2pq}}, \sqrt{\varepsilon_{2}\varepsilon_{p}\varepsilon_{2p}}, \sqrt{  \sqrt{\varepsilon_{2q}}\sqrt{\varepsilon_{pq}\varepsilon_{2pq}}    } \rangle.$$
			
			Furthermore, we have $h_2(2p)=4$, $h_2(pq)=8$,  $h_2(2pq)=4$ and  $h_2(\KK)=8$.
			\item Consider the prime numbers $p=17$ and  $q=103$. In this case, the   conditions of the second item of the above theorem are verified and we have:
			$$E_{\KK}=\langle -1,   \varepsilon_{2}, \varepsilon_{p},   \sqrt{\varepsilon_{q}},   \sqrt{\varepsilon_{pq}} ,\sqrt{ \varepsilon_{2pq}},  \sqrt{\varepsilon_{2p}},  \sqrt{  \sqrt{\varepsilon_{2q}}\sqrt{\varepsilon_{pq}\varepsilon_{2pq}}    }  \rangle.$$
			 	Furthermore, we have $h_2(2p)=2$, $h_2(pq)=4$,  $h_2(2pq)=4$ and  $h_2(\KK)=4$.
		\end{enumerate}
	\end{exam}

	\begin{theorem}\label{T_7_1-C7} Let $p\equiv 1\pmod{8}$ and $q\equiv7\pmod 8$ be two primes such that     $\genfrac(){}{0}{p}{q} =1$. 
		Put     $\KK=\QQ(\sqrt 2, \sqrt{p}, \sqrt{q} )$.  Assume furthermore that $2p(x+1)$ and $v+1$ are squares in $\mathbb{N}$, where $x$ and $v$ are defined in Lemma \ref{lm expressions of units under cond 3 case: 7=1 mod 8 and ()=1}.
		\begin{enumerate}[\rm 1)]
			\item Assume that $N(\varepsilon_{2p})=-1$.  We have
			\begin{enumerate}[\rm $\bullet$]
				
				\item The unit group of $\KK$ is :
				$$E_{\KK}=\langle -1,   \varepsilon_{2}, \varepsilon_{p} ,    \sqrt{\varepsilon_{q}},  \sqrt{\varepsilon_{2q}},\sqrt{ \varepsilon_{pq}} ,\sqrt{ \varepsilon_{2pq}}, \sqrt{\varepsilon_{2}\varepsilon_{p}\varepsilon_{2p}}  \rangle,$$

				\item  The $2$-class number  of  $\KK$ equals $ \frac{1}{2^{4}} h_2(2p)h_2(pq)h_2(2pq)$.  
			\end{enumerate}
			\item Assume that $N(\varepsilon_{2p})=1$ and let $a\in\{0,1\}$ such that $a \equiv 1+u\pmod2$.   We have
			
			\begin{enumerate}[\rm $\bullet$]
				\item  The unit group of $\KK$ is :
					$$E_{\KK}=\langle -1,   \varepsilon_{2}, \varepsilon_{p},   \sqrt{\varepsilon_{q}}, \sqrt{\varepsilon_{2q}},  \sqrt{\varepsilon_{pq}} ,\sqrt{ \varepsilon_{2pq}},    \sqrt{    \varepsilon_{2}^{a\alpha}\varepsilon_{p}^{u\alpha} \sqrt{\varepsilon_{q}}^\alpha\sqrt{\varepsilon_{pq}}^\alpha
					\sqrt{\varepsilon_{2p}}^{1+\gamma}  }  \rangle$$
				where $\alpha$, $\gamma\in \{0,1\}$ such that $\alpha\not=\gamma$ and $\alpha =1$ if and only if $\varepsilon_{2}^{a}\varepsilon_{p}^{u} \sqrt{\varepsilon_{q}}\sqrt{\varepsilon_{pq}}
				\sqrt{\varepsilon_{2p}}$ is a square in $\KK$.
				\item  The $2$-class number  of  $\KK$ equals $ \frac{1}{2^{4-\alpha}} h_2(2p)h_2(pq)h_2(2pq)$.
			\end{enumerate}
		\end{enumerate}
	\end{theorem}
	\begin{proof} 
		\begin{enumerate}[\rm 1)]
			\item 		Assume that   $N(\varepsilon_{2p})=-1$.  By Lemma \ref{units of k1},     $\{\varepsilon_{2}, \varepsilon_{p},	\sqrt{\varepsilon_{2}\varepsilon_{p}\varepsilon_{2p}}\}$  is a  fundamental system of units of $k_1$. Using Lemmas \ref{lm expressions of units q_2q} and \ref{lm expressions of units under cond 3 case: 7=1 mod 8 and ()=1}, we check   that
			$\{\varepsilon_{2}, \sqrt{\varepsilon_{q}}, \sqrt{\varepsilon_{2q}}\}$ and $\{ \varepsilon_{2}, 	{\varepsilon_{ 2pq}},\sqrt{\varepsilon_{pq}}\}$ are respectively fundamental systems of units of $k_2$ and $k_3$.
			It follows that,  	$$E_{k_1}E_{k_2}E_{k_3}=\langle-1,  \varepsilon_{2}, \varepsilon_{p},   \sqrt{\varepsilon_{q}}, \sqrt{\varepsilon_{2q}},   {\varepsilon_{ 2pq}} ,\sqrt{ \varepsilon_{pq}}, \sqrt{\varepsilon_{2}\varepsilon_{p}\varepsilon_{2p}}\rangle.$$	
			Thus we shall determine elements of $E_{k_1}E_{k_2}E_{k_3}$ which are squares in $\KK$. 
			Notice that by Lemma \ref{lm expressions of units under cond 3 case: 7=1 mod 8 and ()=1},
			$\varepsilon_{2pq}$ is a square in  $\KK$.
			Let  $\xi$ is an element of $\KK$ which is the  square root of an element of $E_{k_1}E_{k_2}E_{k_3}$. We can assume that
			$$\xi^2=\varepsilon_{2}^a\varepsilon_{p}^b \sqrt{\varepsilon_{q}}^c\sqrt{\varepsilon_{2q}}^d\sqrt{\varepsilon_{pq}}^f 
			\sqrt{\varepsilon_{2}\varepsilon_{p}\varepsilon_{2p}}^g,$$
			where $a, b, c, d,   f$ and $g$ are in $\{0, 1\}$.

			\noindent\ding{224}  Let us start	by applying   the norm map $N_{\KK/k_2}=1+\tau_2$. We have 
			$\sqrt{\varepsilon_{2}\varepsilon_{p}\varepsilon_{2p}}^{1+ \tau_2}=(-1)^v\varepsilon_{2}$, for some $v\in\{0,1\}$. By means of 
			\eqref{norm q=7 p=1mod4q} and Table \ref{tab3 case: 7=1 mod 8 and ()=1},   we get:
			\begin{eqnarray*}
				N_{\KK/k_2}(\xi^2)&=&
				\varepsilon_{2}^{2a}\cdot(-1)^b \cdot \varepsilon_{q}^c\cdot\varepsilon_{2q}^d\cdot1 \cdot (-1)^{gv}\varepsilon_{2}^{g}\\
				&=&	\varepsilon_{2}^{2a}  \varepsilon_{q}^c\varepsilon_{2q}^d\cdot(-1)^{b+  gv} \varepsilon_{2}^g.
			\end{eqnarray*}
	Thus    $b+  gv= 0\pmod2$ and $g=0$. So $b=0$ and 
	$$\xi^2=\varepsilon_{2}^a  \sqrt{\varepsilon_{q}}^c\sqrt{\varepsilon_{2q}}^d\sqrt{\varepsilon_{pq}}^f .$$
	\noindent\ding{224} Let us apply the norm $N_{\KK/k_5}=1+\tau_1\tau_2$, with $k_5=\QQ(\sqrt{q}, \sqrt{2p})$. By the table \eqref{norm q=7 p=1mod4q} and Table \ref{tab3 case: 7=1 mod 8 and ()=1}, we  have:	
	\begin{eqnarray*}
		N_{\KK/k_5}(\xi^2)&=&(-1)^a  \cdot(-1)^c\cdot \varepsilon_{  q}^c\cdot(-1)^d \cdot(-1)^f   \\
		&=&	 (-1)^{a + c+d +f }\varepsilon_{  q}^c   .
	\end{eqnarray*}
	Thus $c=0$ and so $a  +d+f =0\pmod 2$.   Therefore, 
  $$\xi^2=\varepsilon_{2}^a   \sqrt{\varepsilon_{2q}}^d\sqrt{\varepsilon_{pq}}^f .$$

\noindent\ding{224} Let us apply the norm $N_{\KK/k_6}=1+\tau_1\tau_3$, with $k_6=\QQ(\sqrt{p}, \sqrt{2q})$. By the table \eqref{norm q=7 p=1mod4q} and Table \ref{tab3 case: 7=1 mod 8 and ()=1}, we  have:
\begin{eqnarray*}
	N_{\KK/k_6}(\xi^2)&=&(-1)^a  \cdot (-1)^d\cdot \varepsilon_{2  q}^d\cdot  (-1)^f  \\
	&=&	    (-1)^{a+d+f}\varepsilon_{2  q}^d .
\end{eqnarray*}	
Thus  $d =0$ and $a=f$.	Therefore,
$$\xi^2=\varepsilon_{2}^a    \sqrt{\varepsilon_{pq}}^a .$$

\noindent\ding{224} Let us apply the norm $N_{\KK/k_4}=1+\tau_1$, with $k_4=\QQ(\sqrt{p}, \sqrt{q})$.  By the table \eqref{norm q=7 p=1mod4q} and Table \ref{tab3 case: 7=1 mod 8 and ()=1}, we  have:		
\begin{eqnarray*}
	N_{\KK/k_4}(\xi^2)&=&(-1)^a\cdot (-1)^a\cdot \varepsilon_{  pq}^a=\varepsilon_{  pq}^a.
\end{eqnarray*}
Thus   $a=0$. So we have the first item.

\item 	Assume that   $N(\varepsilon_{2p})=1$. So we have:
$$E_{k_1}E_{k_2}E_{k_3}=\langle-1,  \varepsilon_{2}, \varepsilon_{p},   \sqrt{\varepsilon_{q}}, \sqrt{\varepsilon_{2q}},   {\varepsilon_{ 2pq}} ,\sqrt{ \varepsilon_{pq}}, \sqrt{\varepsilon_{2p}}\rangle.$$		
To determine the elements of $E_{k_1}E_{k_2}E_{k_3}$ which are squares in $\KK$, let us consider $\xi$   an element of $\KK$ which is the  square root of an element of $E_{k_1}E_{k_2}E_{k_3}$. As $\varepsilon_{2pq}$ is a square in  $\KK$, we can assume that
$$\xi^2=\varepsilon_{2}^a\varepsilon_{p}^b \sqrt{\varepsilon_{q}}^c\sqrt{\varepsilon_{2q}}^d\sqrt{\varepsilon_{pq}}^f 
\sqrt{\varepsilon_{2p}}^g,$$
where $a, b, c, d,   f$ and $g$ are in $\{0, 1\}$.

\noindent\ding{224}  Let us start	by applying   the norm map $N_{\KK/k_2}=1+\tau_2$.  By \eqref{T_3_-1_eqi2p_N=1}, 
\eqref{norm q=7 p=1mod4q} and Table \ref{tab3 case: 7=1 mod 8 and ()=1},   we have:
\begin{eqnarray*}
	N_{\KK/k_2}(\xi^2)&=&
	\varepsilon_{2}^{2a}\cdot(-1)^b \cdot \varepsilon_{q}^c\cdot\varepsilon_{2q}^d\cdot1 \cdot (-1)^{gu} \\
	&=&	\varepsilon_{2}^{2a}  \varepsilon_{q}^c\varepsilon_{2q}^d\cdot(-1)^{b+ gu}  .
\end{eqnarray*}

Thus    $b+  gu= 0\pmod2$. So $b=gu$.

	\noindent\ding{224} Let us apply the norm $N_{\KK/k_5}=1+\tau_1\tau_2$, with $k_5=\QQ(\sqrt{q}, \sqrt{2p})$. By the table \eqref{norm q=7 p=1mod4q} and Table \ref{tab3 case: 7=1 mod 8 and ()=1}, we  have:	
\begin{eqnarray*}
	N_{\KK/k_5}(\xi^2)&=&(-1)^a \cdot (-1)^b \cdot(-1)^c\cdot \varepsilon_{  q}^c\cdot(-1)^d\cdot(-1)^f \cdot (-1)^g \cdot\varepsilon_{2p}^g\\
	&=&	 (-1)^{a+b+ c+d+f+g  }\varepsilon_{  q}^c \varepsilon_{2p}^g  .
\end{eqnarray*}
Thus $c=g$ and so $a+b+  d+f      =0\pmod 2$.   Therefore, 
$$\xi^2=\varepsilon_{2}^a\varepsilon_{p}^b \sqrt{\varepsilon_{q}}^c\sqrt{\varepsilon_{2q}}^d\sqrt{\varepsilon_{pq}}^f
\sqrt{\varepsilon_{2p}}^c.$$
\noindent\ding{224} Let us apply the norm $N_{\KK/k_6}=1+\tau_1\tau_3$, with $k_6=\QQ(\sqrt{p}, \sqrt{2q})$. By the table \eqref{norm q=7 p=1mod4q} and Table \ref{tab3 case: 7=1 mod 8 and ()=1}, we  have:
\begin{eqnarray*}
	N_{\KK/k_6}(\xi^2)&=&(-1)^a\cdot \varepsilon_{p}^{2b} \cdot (-1)^c  \cdot (-1)^d\cdot \varepsilon_{2  q}^d \cdot (-1)^f\cdot  (-1)^{cu+c}\\
	&=&	   \varepsilon_{p}^{2b} (-1)^{a+d+f+cu}\varepsilon_{2  q}^d.
\end{eqnarray*}	
Thus $d=0$ and so $a+f+cu=0\pmod 2$. As $a+b+  d+f      =0\pmod 2$, then $a+b+  f      =0\pmod 2$.
Therefore,
$$\xi^2=\varepsilon_{2}^a\varepsilon_{p}^b \sqrt{\varepsilon_{q}}^c \sqrt{\varepsilon_{pq}}^f
\sqrt{\varepsilon_{2p}}^c.$$	
	\noindent\ding{224} Let us apply the norm $N_{\KK/k_4}=1+\tau_1$, with $k_4=\QQ(\sqrt{p}, \sqrt{q})$.  By the table \eqref{norm q=7 p=1mod4q} and Table \ref{tab3 case: 7=1 mod 8 and ()=1}, we  have:		
\begin{eqnarray*}
	N_{\KK/k_4}(\xi^2)&=&(-1)^{a}\cdot \varepsilon_{p}^{2b}\cdot (-1)^c\cdot \varepsilon_{q}^{c} \cdot (-1)^f\cdot \varepsilon_{  pq}^f\cdot (-1)^{cu+c}\\
	&=& \varepsilon_{p}^{2b}(-1)^{a+f+cu}\varepsilon_{q}^{c}\varepsilon_{  pq}^f .
\end{eqnarray*}
Thus   $c=f$ and  $a+c+cu=0\pmod 2$. Therefore, 
$$\xi^2=\varepsilon_{2}^a\varepsilon_{p}^{cu} \sqrt{\varepsilon_{q}}^c \sqrt{\varepsilon_{pq}}^c
\sqrt{\varepsilon_{2p}}^c.$$

	\end{enumerate}	
		\end{proof}

	\begin{exam}Using  PARI/GP calculator version  2.15.5 (64bit), Feb 11 2024, we  get the following examples which illustrate the above theorem.
		\begin{enumerate}[\rm $1)$]
			\item  Consider the prime numbers $p=41$ and  $q=103$. In this case, the   conditions of the first item of the above theorem are verified and we have:
				$$E_{\KK}=\langle -1,   \varepsilon_{2}, \varepsilon_{p} ,    \sqrt{\varepsilon_{q}},  \sqrt{\varepsilon_{2q}},\sqrt{ \varepsilon_{pq}} ,\sqrt{ \varepsilon_{2pq}}, \sqrt{\varepsilon_{2}\varepsilon_{p}\varepsilon_{2p}}  \rangle.$$
			
			Furthermore, we have $h_2(2p)=4$, $h_2(pq)=4$,  $h_2(2pq)=4$ and  $h_2(\KK)=4$.
			\item Consider the prime numbers $p=17$ and  $q=359$. In this case, the   conditions of the second item of the above theorem are verified and we have:
			$$E_{\KK}=\langle -1,   \varepsilon_{2}, \varepsilon_{p},   \sqrt{\varepsilon_{q}}, \sqrt{\varepsilon_{2q}},  \sqrt{\varepsilon_{pq}} ,\sqrt{ \varepsilon_{2pq}},    \sqrt{    \varepsilon_{2}^{a}\varepsilon_{p}^{u} \sqrt{\varepsilon_{q}}\sqrt{\varepsilon_{pq}}
				\sqrt{\varepsilon_{2p}}  }  \rangle,$$
			for some 	$a\in\{0,1\}$ such that $a \equiv 1+u\pmod2$.	Furthermore, we have $h_2(2p)=2$, $h_2(pq)=4$,  $h_2(2pq)=8$ and  $h_2(\KK)=8$.
		\end{enumerate}
	\end{exam}

\begin{theorem}\label{T_7_1-C8} Let $p\equiv 1\pmod{8}$ and $q\equiv7\pmod 8$ be two primes such that     $\genfrac(){}{0}{p}{q} =1$. 
	Put     $\KK=\QQ(\sqrt 2, \sqrt{p}, \sqrt{q} )$.  Assume furthermore that $2p(x+1)$ and $p(v+1)$ are squares in $\mathbb{N}$, where $x$ and $v$ are defined in Lemma \ref{lm expressions of units under cond 3 case: 7=1 mod 8 and ()=1}.
	\begin{enumerate}[\rm 1)]
		\item Assume that $N(\varepsilon_{2p})=-1$.  We have
		\begin{enumerate}[\rm $\bullet$]
			
			\item The unit group of $\KK$ is :
			$$E_{\KK}=\langle -1,   \varepsilon_{2}, \varepsilon_{p} ,      \sqrt{\varepsilon_{2q}},\sqrt{ \varepsilon_{pq}} ,\sqrt{ \varepsilon_{2pq}}, \sqrt{\varepsilon_{2}\varepsilon_{p}\varepsilon_{2p}} ,\sqrt{   \sqrt{\varepsilon_{q}}^{1+\gamma}  \sqrt{\varepsilon_{pq}\varepsilon_{ 2pq}}^{\alpha}  } \rangle,$$
			where $\alpha$, $\gamma\in \{0,1\}$ such that $\alpha\not=\gamma$ and $\alpha =1$ if and only if $\sqrt{\varepsilon_{q}}   \sqrt{\varepsilon_{pq}\varepsilon_{ 2pq}} $ is a square in $\KK$.
			
			\item  The $2$-class number  of  $\KK$ equals $ \frac{1}{2^{4-\alpha}} h_2(2p)h_2(pq)h_2(2pq)$.  
		\end{enumerate}
		\item Assume that $N(\varepsilon_{2p})=1$.   We have
		
		\begin{enumerate}[\rm $\bullet$]
			\item  The unit group of $\KK$ is :
			$$E_{\KK}=\langle -1,   \varepsilon_{2}, \varepsilon_{p},   \sqrt{\varepsilon_{2q}},  \sqrt{\varepsilon_{pq}} ,\sqrt{ \varepsilon_{2pq}}, \sqrt{\varepsilon_{2p}},    \sqrt{     \sqrt{\varepsilon_{q}}^{1+\gamma} \sqrt{\varepsilon_{pq}\varepsilon_{2pq}}^\alpha   }  \rangle$$
			where $\alpha$, $\gamma\in \{0,1\}$ such that $\alpha\not=\gamma$ and $\alpha =1$ if and only if $ \sqrt{\varepsilon_{q}}  \sqrt{\varepsilon_{pq}\varepsilon_{2pq}}  $ is a square in $\KK$.
			\item  The $2$-class number  of  $\KK$ equals $ \frac{1}{2^{4-\alpha}} h_2(2p)h_2(pq)h_2(2pq)$.
		\end{enumerate}
	\end{enumerate}
\end{theorem}
\begin{proof} 
	\begin{enumerate}[\rm 1)]
		\item 		Assume that   $N(\varepsilon_{2p})=-1$.  By Lemma \ref{units of k1},     $\{\varepsilon_{2}, \varepsilon_{p},	\sqrt{\varepsilon_{2}\varepsilon_{p}\varepsilon_{2p}}\}$  is a  fundamental system of units of $k_1$. Using Lemmas \ref{lm expressions of units q_2q} and \ref{lm expressions of units under cond 3 case: 7=1 mod 8 and ()=1}, we check   that
		$\{\varepsilon_{2}, \sqrt{\varepsilon_{q}}, \sqrt{\varepsilon_{2q}}\}$ and $\{ \varepsilon_{2}, 	{\varepsilon_{ pq}},\sqrt{\varepsilon_{pq}\varepsilon_{ 2pq}}\}$ are respectively fundamental systems of units of $k_2$ and $k_3$.
		It follows that,  	$$E_{k_1}E_{k_2}E_{k_3}=\langle-1,  \varepsilon_{2}, \varepsilon_{p},   \sqrt{\varepsilon_{q}}, \sqrt{\varepsilon_{2q}},   {\varepsilon_{ pq}} ,\sqrt{ \varepsilon_{pq}\varepsilon_{ 2pq}}, \sqrt{\varepsilon_{2}\varepsilon_{p}\varepsilon_{2p}}\rangle.$$	
		Thus we shall determine elements of $E_{k_1}E_{k_2}E_{k_3}$ which are squares in $\KK$. 
		Notice that by Lemma \ref{lm expressions of units under cond 3 case: 7=1 mod 8 and ()=1},
		$\varepsilon_{pq}$ is a square in  $\KK$.
		Let  $\xi$ is an element of $\KK$ which is the  square root of an element of $E_{k_1}E_{k_2}E_{k_3}$. We can assume that
		$$\xi^2=\varepsilon_{2}^a\varepsilon_{p}^b \sqrt{\varepsilon_{q}}^c\sqrt{\varepsilon_{2q}}^d\sqrt{\varepsilon_{pq}\varepsilon_{ 2pq}}^f 
		\sqrt{\varepsilon_{2}\varepsilon_{p}\varepsilon_{2p}}^g,$$
		where $a, b, c, d,   f$ and $g$ are in $\{0, 1\}$.

		\noindent\ding{224}  Let us start	by applying   the norm map $N_{\KK/k_2}=1+\tau_2$. We have 
		$\sqrt{\varepsilon_{2}\varepsilon_{p}\varepsilon_{2p}}^{1+ \tau_2}=(-1)^v\varepsilon_{2}$, for some $v\in\{0,1\}$. By means of 
		\eqref{norm q=7 p=1mod4q} and Table \ref{tab3 case: 7=1 mod 8 and ()=1},   we get:
		\begin{eqnarray*}
			N_{\KK/k_2}(\xi^2)&=&
			\varepsilon_{2}^{2a}\cdot(-1)^b \cdot \varepsilon_{q}^c\cdot\varepsilon_{2q}^d\cdot1 \cdot (-1)^{gv}\varepsilon_{2}^{g}\\
			&=&	\varepsilon_{2}^{2a}  \varepsilon_{q}^c\varepsilon_{2q}^d\cdot(-1)^{b+  gv} \varepsilon_{2}^g.
		\end{eqnarray*}
		Thus    $b+  gv= 0\pmod2$ and $g=0$. Therefore, $b=0$ and 
	$$\xi^2=\varepsilon_{2}^a \sqrt{\varepsilon_{q}}^c\sqrt{\varepsilon_{2q}}^d\sqrt{\varepsilon_{pq}\varepsilon_{ 2pq}}^f.$$
	
	\noindent\ding{224} Let us apply the norm $N_{\KK/k_5}=1+\tau_1\tau_2$, with $k_5=\QQ(\sqrt{q}, \sqrt{2p})$. By the table \eqref{norm q=7 p=1mod4q} and Table \ref{tab3 case: 7=1 mod 8 and ()=1}, we  have:	
	\begin{eqnarray*}
		N_{\KK/k_5}(\xi^2)&=&(-1)^a  \cdot(-1)^c\cdot \varepsilon_{  q}^c\cdot(-1)^d \cdot(-1)^f\cdot \varepsilon_{ 2pq}^f   \\
		&=&	 (-1)^{a + c+d +f }\varepsilon_{  q}^c \varepsilon_{ 2pq}^f  .
	\end{eqnarray*}
	Thus $c=f$ and   $a + c+d +f =0\pmod 2$.   Therefore, $a  =d$ and
	$$\xi^2=\varepsilon_{2}^a \sqrt{\varepsilon_{q}}^c\sqrt{\varepsilon_{2q}}^a\sqrt{\varepsilon_{pq}\varepsilon_{ 2pq}}^c.$$

\noindent\ding{224} Let us apply the norm $N_{\KK/k_6}=1+\tau_1\tau_3$, with $k_6=\QQ(\sqrt{p}, \sqrt{2q})$. By the table \eqref{norm q=7 p=1mod4q} and Table \ref{tab3 case: 7=1 mod 8 and ()=1}, we  have:
\begin{eqnarray*}
	N_{\KK/k_6}(\xi^2)&=&(-1)^a  \cdot(-1)^c  \cdot (-1)^a\cdot \varepsilon_{2  q}^a\cdot  (-1)^c\cdot \varepsilon_{ 2pq}^c \\
	&=&	   \varepsilon_{ 2pq}^c  \cdot\varepsilon_{2  q}^a .
\end{eqnarray*}	
Thus  $a =0$.	Therefore,	
		$$\xi^2=  \sqrt{\varepsilon_{q}}^c \sqrt{\varepsilon_{pq}\varepsilon_{ 2pq}}^c.$$

\noindent\ding{224} Let us apply the norm $N_{\KK/k_4}=1+\tau_1$, with $k_4=\QQ(\sqrt{p}, \sqrt{q})$.  By the table \eqref{norm q=7 p=1mod4q} and Table \ref{tab3 case: 7=1 mod 8 and ()=1}, we  have:		
\begin{eqnarray*}
	N_{\KK/k_4}(\xi^2)&=&(-1)^c\cdot  \varepsilon_{  q}^c\cdot (-1)^c\cdot \varepsilon_{  pq}^c=\varepsilon_{  q}^c   \varepsilon_{  pq}^c.
\end{eqnarray*}
We deduce nothing.

	\item Assume that   $N(\varepsilon_{2p})=1$. So we have:
$$E_{k_1}E_{k_2}E_{k_3}=\langle-1,  \varepsilon_{2}, \varepsilon_{p},   \sqrt{\varepsilon_{q}}, \sqrt{\varepsilon_{2q}},   {\varepsilon_{ pq}} ,\sqrt{ \varepsilon_{pq}\varepsilon_{ 2pq}}, \sqrt{\varepsilon_{2p}}\rangle.$$		
	To determine the elements of $E_{k_1}E_{k_2}E_{k_3}$ which are squares in $\KK$, let us consider $\xi$   an element of $\KK$ which is the  square root of an element of $E_{k_1}E_{k_2}E_{k_3}$. As $\varepsilon_{pq}$ is a square in  $\KK$, we can assume that
	$$\xi^2=\varepsilon_{2}^a\varepsilon_{p}^b \sqrt{\varepsilon_{q}}^c\sqrt{\varepsilon_{2q}}^d\sqrt{\varepsilon_{pq}\varepsilon_{ 2pq}}^f 
	\sqrt{ \varepsilon_{2p}}^g,$$
	where $a, b, c, d,   f$ and $g$ are in $\{0, 1\}$.
	
	\noindent\ding{224}  Let us start	by applying   the norm map $N_{\KK/k_2}=1+\tau_2$.  By \eqref{T_3_-1_eqi2p_N=1}, 
	\eqref{norm q=7 p=1mod4q} and Table \ref{tab3 case: 7=1 mod 8 and ()=1},   we have:
	\begin{eqnarray*}
		N_{\KK/k_2}(\xi^2)&=&
		\varepsilon_{2}^{2a}\cdot(-1)^b \cdot \varepsilon_{q}^c\cdot\varepsilon_{2q}^d\cdot1 \cdot (-1)^{gu} \\
		&=&	\varepsilon_{2}^{2a}  \varepsilon_{q}^c\varepsilon_{2q}^d\cdot(-1)^{b+ gu}  .
	\end{eqnarray*}
	
	Thus    $b+  gu= 0\pmod2$. So $b=gu$.
	
	\noindent\ding{224} Let us apply the norm $N_{\KK/k_5}=1+\tau_1\tau_2$, with $k_5=\QQ(\sqrt{q}, \sqrt{2p})$. By the table \eqref{norm q=7 p=1mod4q} and Table \ref{tab3 case: 7=1 mod 8 and ()=1}, we  have:	
\begin{eqnarray*}
	N_{\KK/k_5}(\xi^2)&=&(-1)^a \cdot (-1)^b \cdot(-1)^c\cdot \varepsilon_{  q}^c\cdot(-1)^d\cdot(-1)^f \cdot \varepsilon_{2pq}^f \cdot(-1)^g \cdot\varepsilon_{2p}^g\\
	&=&	 (-1)^{a+b+ c+d+f+g  }\varepsilon_{  q}^c\varepsilon_{2pq}^f \varepsilon_{2p}^g  .
\end{eqnarray*}
Thus   $a+b+ c+d+f+g      =0\pmod 2$ and $c+f+  g      =0\pmod 2$.   Therefore, $a+b+  d        =0\pmod 2$.
 
 \noindent\ding{224} Let us apply the norm $N_{\KK/k_6}=1+\tau_1\tau_3$, with $k_6=\QQ(\sqrt{p}, \sqrt{2q})$. By the table \eqref{norm q=7 p=1mod4q} and Table \ref{tab3 case: 7=1 mod 8 and ()=1}, we  have:
 \begin{eqnarray*}
 	N_{\KK/k_6}(\xi^2)&=&(-1)^a\cdot \varepsilon_{p}^{2b} \cdot (-1)^c  \cdot (-1)^d\cdot \varepsilon_{2  q}^d \cdot (-1)^f\cdot \varepsilon_{2pq}^f\cdot (-1)^{gu+g}\\
 	&=&	   \varepsilon_{p}^{2b}\varepsilon_{2pq}^f (-1)^{a+c+d+f+gu+g}\varepsilon_{2  q}^d.
 \end{eqnarray*}	
 Thus $d=0$ and so $a+c +f+gu+g=0\pmod 2$. 
 As $a+b+  d        =0\pmod 2$, then $a=b$.  
 Therefore,
 $$\xi^2=\varepsilon_{2}^a\varepsilon_{p}^a \sqrt{\varepsilon_{q}}^c\sqrt{\varepsilon_{pq}\varepsilon_{ 2pq}}^f \sqrt{ \varepsilon_{2p}}^g.$$
 
\noindent\ding{224} Let us apply the norm $N_{\KK/k_4}=1+\tau_1$, with $k_4=\QQ(\sqrt{p}, \sqrt{q})$.  By the table \eqref{norm q=7 p=1mod4q} and Table \ref{tab3 case: 7=1 mod 8 and ()=1}, we  have:		
\begin{eqnarray*}
	N_{\KK/k_4}(\xi^2)&=&(-1)^{a}\cdot \varepsilon_{p}^{2a}\cdot (-1)^c\cdot \varepsilon_{q}^{c} \cdot (-1)^f\cdot \varepsilon_{  pq}^f\cdot (-1)^{gu+g}\\
	&=& \varepsilon_{p}^{2b}(-1)^{a+c+f+gu+g}\varepsilon_{q}^{c}\varepsilon_{  pq}^f .
\end{eqnarray*}
Thus      $a+c+f+gu+g=0\pmod 2$ and $c=f$. So $a+ gu+g=0\pmod 2$. As  $ c+f +g=0\pmod 2$, then $g=0$ and so $a=0$. Therefore, 
 $$\xi^2=  \sqrt{\varepsilon_{q}}^c\sqrt{\varepsilon_{pq}\varepsilon_{ 2pq}}^c  .$$
 
 So the second item.

  	\end{enumerate}	
\end{proof}

		\begin{exam}Using  PARI/GP calculator version  2.15.5 (64bit), Feb 11 2024, we  get the following examples which illustrate the above theorem.
		\begin{enumerate}[\rm $1)$]
			\item  Consider the prime numbers $p=313$ and  $q=151$. In this case, the   conditions of the first item of the above theorem are verified and we have:
				$$E_{\KK}=\langle -1,   \varepsilon_{2}, \varepsilon_{p} ,      \sqrt{\varepsilon_{2q}},\sqrt{ \varepsilon_{pq}} ,\sqrt{ \varepsilon_{2pq}}, \sqrt{\varepsilon_{2}\varepsilon_{p}\varepsilon_{2p}} ,\sqrt{   \sqrt{\varepsilon_{q}}  \sqrt{\varepsilon_{pq}\varepsilon_{ 2pq}}  } \rangle.$$
		 	Furthermore, we have $h_2(2p)=4$, $h_2(pq)=4$,  $h_2(2pq)=8$ and  $h_2(\KK)=16$.
			\item  Consider the prime numbers $p=17$ and  $q=127$. In this case, the   conditions of the second item of the above theorem are verified and we have:
			$$E_{\KK}=\langle -1,   \varepsilon_{2}, \varepsilon_{p},   \sqrt{\varepsilon_{2q}},  \sqrt{\varepsilon_{pq}} ,\sqrt{ \varepsilon_{2pq}}, \sqrt{\varepsilon_{2p}},    \sqrt{     \sqrt{\varepsilon_{q}} \sqrt{\varepsilon_{pq}\varepsilon_{2pq}}   }  \rangle.$$
			 	Furthermore, we have $h_2(2p)=2$, $h_2(pq)=4$,  $h_2(2pq)=4$ and  $h_2(\KK)=4$.
		\end{enumerate}
	\end{exam}

	\begin{theorem}\label{T_7_1-C9} Let $p\equiv 1\pmod{8}$ and $q\equiv7\pmod 8$ be two primes such that     $\genfrac(){}{0}{p}{q} =1$. 
		Put     $\KK=\QQ(\sqrt 2, \sqrt{p}, \sqrt{q} )$.  Assume furthermore that $2p(x+1)$ and $2p(v+1)$ are squares in $\mathbb{N}$, where $x$ and $v$ are defined in Lemma \ref{lm expressions of units under cond 3 case: 7=1 mod 8 and ()=1}.
		\begin{enumerate}[\rm 1)]
			\item Assume that $N(\varepsilon_{2p})=-1$.  We have
			\begin{enumerate}[\rm $\bullet$]
				
				\item The unit group of $\KK$ is :
				$$E_{\KK}=\langle -1,   \varepsilon_{2}, \varepsilon_{p} ,     \sqrt{\varepsilon_{ q}}, \sqrt{\varepsilon_{2q}},\sqrt{ \varepsilon_{pq}} ,\sqrt{ \varepsilon_{2pq}}, \sqrt{\varepsilon_{2}\varepsilon_{p}\varepsilon_{2p}}   \rangle,$$
				
			%	$$E_{\KK}=\langle -1,   \varepsilon_{2}, \varepsilon_{p} ,      \sqrt{\varepsilon_{2q}},\sqrt{ \varepsilon_{pq}} ,\sqrt{ \varepsilon_{2pq}}, \sqrt{\varepsilon_{2}\varepsilon_{p}\varepsilon_{2p}} ,\sqrt{   \sqrt{\varepsilon_{q}}^{1+\gamma}  \sqrt{\varepsilon_{pq}\varepsilon_{ 2pq}}^{\alpha}  } \rangle,$$
			%	where $\alpha$, $\gamma\in \{0,1\}$ such that $\alpha\not=\gamma$ and $\alpha =1$ if and only if $\sqrt{\varepsilon_{q}}   \sqrt{\varepsilon_{pq}\varepsilon_{ 2pq}} $ is a square in $\KK$.
				
				\item  The $2$-class number  of  $\KK$ equals $ \frac{1}{2^{4}} h_2(2p)h_2(pq)h_2(2pq)$.  
			\end{enumerate}
			\item Assume that $N(\varepsilon_{2p})=1$ and let $a\in\{0,1\}$ such that $a \equiv 1+u\pmod2$.   We have
			
			\begin{enumerate}[\rm $\bullet$]
				\item  The unit group of $\KK$ is :
				$$E_{\KK}=\langle -1,   \varepsilon_{2}, \varepsilon_{p},  \sqrt{\varepsilon_{q}}, \sqrt{\varepsilon_{2q}},  \sqrt{\varepsilon_{pq}} ,\sqrt{ \varepsilon_{2pq}},      \sqrt{     \varepsilon_{2}^{a\alpha}\varepsilon_{p}^{ u\alpha}   \sqrt{\varepsilon_{pq}\varepsilon_{ 2pq}}^{\alpha} \sqrt{ \varepsilon_{2p}}^{1+\gamma}   }  \rangle$$
				where $\alpha$, $\gamma\in \{0,1\}$ such that $\alpha\not=\gamma$ and $\alpha =1$ if and only if $ \varepsilon_{2}^{a}\varepsilon_{p}^{ u}   \sqrt{\varepsilon_{pq}\varepsilon_{ 2pq}} \sqrt{ \varepsilon_{2p}}  $ is a square in $\KK$.
				\item  The $2$-class number  of  $\KK$ equals $ \frac{1}{2^{4-\alpha}} h_2(2p)h_2(pq)h_2(2pq)$.
			\end{enumerate}
		\end{enumerate}
	\end{theorem}
	\begin{proof} 
		\begin{enumerate}[\rm 1)]
			\item 		Assume that   $N(\varepsilon_{2p})=-1$.  By Lemma \ref{units of k1},     $\{\varepsilon_{2}, \varepsilon_{p},	\sqrt{\varepsilon_{2}\varepsilon_{p}\varepsilon_{2p}}\}$  is a  fundamental system of units of $k_1$. Using Lemmas \ref{lm expressions of units q_2q} and \ref{lm expressions of units under cond 3 case: 7=1 mod 8 and ()=1}, we check   that
			$\{\varepsilon_{2}, \sqrt{\varepsilon_{q}}, \sqrt{\varepsilon_{2q}}\}$ and $\{ \varepsilon_{2}, 	{\varepsilon_{ pq}},\sqrt{\varepsilon_{pq}\varepsilon_{ 2pq}}\}$ are respectively fundamental systems of units of $k_2$ and $k_3$.
			It follows that,  	$$E_{k_1}E_{k_2}E_{k_3}=\langle-1,  \varepsilon_{2}, \varepsilon_{p},   \sqrt{\varepsilon_{q}}, \sqrt{\varepsilon_{2q}},   {\varepsilon_{ pq}} ,\sqrt{ \varepsilon_{pq}\varepsilon_{ 2pq}}, \sqrt{\varepsilon_{2}\varepsilon_{p}\varepsilon_{2p}}\rangle.$$	
			Thus we shall determine elements of $E_{k_1}E_{k_2}E_{k_3}$ which are squares in $\KK$. 
			Notice that by Lemma \ref{lm expressions of units under cond 3 case: 7=1 mod 8 and ()=1},
			$\varepsilon_{pq}$ is a square in  $\KK$.
			Let  $\xi$ is an element of $\KK$ which is the  square root of an element of $E_{k_1}E_{k_2}E_{k_3}$. We can assume that
			$$\xi^2=\varepsilon_{2}^a\varepsilon_{p}^b \sqrt{\varepsilon_{q}}^c\sqrt{\varepsilon_{2q}}^d\sqrt{\varepsilon_{pq}\varepsilon_{ 2pq}}^f 
			\sqrt{\varepsilon_{2}\varepsilon_{p}\varepsilon_{2p}}^g,$$
			where $a, b, c, d,   f$ and $g$ are in $\{0, 1\}$.

			\noindent\ding{224}  Let us start	by applying   the norm map $N_{\KK/k_2}=1+\tau_2$. We have 
			$\sqrt{\varepsilon_{2}\varepsilon_{p}\varepsilon_{2p}}^{1+ \tau_2}=(-1)^v\varepsilon_{2}$, for some $v\in\{0,1\}$. By means of 
			\eqref{norm q=7 p=1mod4q} and Table \ref{tab3 case: 7=1 mod 8 and ()=1},   we get:
			\begin{eqnarray*}
				N_{\KK/k_2}(\xi^2)&=&
				\varepsilon_{2}^{2a}\cdot(-1)^b \cdot \varepsilon_{q}^c\cdot\varepsilon_{2q}^d\cdot1 \cdot (-1)^{gv}\varepsilon_{2}^{g}\\
				&=&	\varepsilon_{2}^{2a}  \varepsilon_{q}^c\varepsilon_{2q}^d\cdot(-1)^{b+  gv} \varepsilon_{2}^g.
			\end{eqnarray*}
			Thus    $b+  gv= 0\pmod2$ and $g=0$. Therefore, $b=0$ and 
			$$\xi^2=\varepsilon_{2}^a \sqrt{\varepsilon_{q}}^c\sqrt{\varepsilon_{2q}}^d\sqrt{\varepsilon_{pq}\varepsilon_{ 2pq}}^f.$$
	
		\noindent\ding{224} Let us apply the norm $N_{\KK/k_5}=1+\tau_1\tau_2$, with $k_5=\QQ(\sqrt{q}, \sqrt{2p})$. By the table \eqref{norm q=7 p=1mod4q} and Table \ref{tab3 case: 7=1 mod 8 and ()=1}, we  have:	
	\begin{eqnarray*}
		N_{\KK/k_5}(\xi^2)&=&(-1)^a  \cdot(-1)^c\cdot \varepsilon_{  q}^c\cdot(-1)^d \cdot \varepsilon_{ 2pq}^f   \\
		&=&	 (-1)^{a + c+d   }\varepsilon_{  q}^c \varepsilon_{ 2pq}^f  .
	\end{eqnarray*}
	Thus $c=f$ and   $a + c+d   =0\pmod 2$.   Therefore,  
			$$\xi^2=\varepsilon_{2}^a \sqrt{\varepsilon_{q}}^c\sqrt{\varepsilon_{2q}}^d\sqrt{\varepsilon_{pq}\varepsilon_{ 2pq}}^c.$$
	
	\noindent\ding{224} Let us apply the norm $N_{\KK/k_6}=1+\tau_1\tau_3$, with $k_6=\QQ(\sqrt{p}, \sqrt{2q})$. By the table \eqref{norm q=7 p=1mod4q} and Table \ref{tab3 case: 7=1 mod 8 and ()=1}, we  have:
	\begin{eqnarray*}
		N_{\KK/k_6}(\xi^2)&=&(-1)^a  \cdot(-1)^c  \cdot (-1)^d\cdot \varepsilon_{2  q}^d\cdot  \varepsilon_{ 2pq}^c \\
		&=&	   \varepsilon_{ 2pq}^c(-1)^{a + c+d   }  \cdot\varepsilon_{2  q}^d .
	\end{eqnarray*}	
	Thus  $d =0$ and $a + c+d   =0\pmod 2$.	Therefore,	 $a=c$ and 
			$$\xi^2=\varepsilon_{2}^a \sqrt{\varepsilon_{q}}^a \sqrt{\varepsilon_{pq}\varepsilon_{ 2pq}}^a.$$
	
		\noindent\ding{224} Let us apply the norm $N_{\KK/k_4}=1+\tau_1$, with $k_4=\QQ(\sqrt{p}, \sqrt{q})$.  By the table \eqref{norm q=7 p=1mod4q} and Table \ref{tab3 case: 7=1 mod 8 and ()=1}, we  have:		
	\begin{eqnarray*}
		N_{\KK/k_4}(\xi^2)&=&(-1)^a\cdot(-1)^a\cdot  \varepsilon_{  q}^a\cdot  \varepsilon_{  pq}^a=  \varepsilon_{  pq}^a\cdot \varepsilon_{  q}^a .
	\end{eqnarray*}
 Thus, $a=0$. So we have the first item.
	
	\item Assume that   $N(\varepsilon_{2p})=1$. So we have:
$$E_{k_1}E_{k_2}E_{k_3}=\langle-1,  \varepsilon_{2}, \varepsilon_{p},   \sqrt{\varepsilon_{q}}, \sqrt{\varepsilon_{2q}},   {\varepsilon_{ pq}} ,\sqrt{ \varepsilon_{pq}\varepsilon_{ 2pq}}, \sqrt{\varepsilon_{2p}}\rangle.$$		
To determine the elements of $E_{k_1}E_{k_2}E_{k_3}$ which are squares in $\KK$, let us consider $\xi$   an element of $\KK$ which is the  square root of an element of $E_{k_1}E_{k_2}E_{k_3}$. As $\varepsilon_{pq}$ is a square in  $\KK$, we can assume that
$$\xi^2=\varepsilon_{2}^a\varepsilon_{p}^b \sqrt{\varepsilon_{q}}^c\sqrt{\varepsilon_{2q}}^d\sqrt{\varepsilon_{pq}\varepsilon_{ 2pq}}^f 
\sqrt{ \varepsilon_{2p}}^g,$$
where $a, b, c, d,   f$ and $g$ are in $\{0, 1\}$.

\noindent\ding{224}  Let us start	by applying   the norm map $N_{\KK/k_2}=1+\tau_2$.  By \eqref{T_3_-1_eqi2p_N=1}, 
\eqref{norm q=7 p=1mod4q} and Table \ref{tab3 case: 7=1 mod 8 and ()=1},   we have:
\begin{eqnarray*}
	N_{\KK/k_2}(\xi^2)&=&
	\varepsilon_{2}^{2a}\cdot(-1)^b \cdot \varepsilon_{q}^c\cdot\varepsilon_{2q}^d\cdot1 \cdot (-1)^{gu} \\
	&=&	\varepsilon_{2}^{2a}  \varepsilon_{q}^c\varepsilon_{2q}^d\cdot(-1)^{b+ gu}  .
\end{eqnarray*}

Thus    $b+  gu= 0\pmod2$. So $b=gu$.

\noindent\ding{224} Let us apply the norm $N_{\KK/k_5}=1+\tau_1\tau_2$, with $k_5=\QQ(\sqrt{q}, \sqrt{2p})$. By the table \eqref{norm q=7 p=1mod4q} and Table \ref{tab3 case: 7=1 mod 8 and ()=1}, we  have:	
\begin{eqnarray*}
	N_{\KK/k_5}(\xi^2)&=&(-1)^a \cdot (-1)^b \cdot(-1)^c\cdot \varepsilon_{  q}^c\cdot(-1)^d  \cdot \varepsilon_{2pq}^f \cdot(-1)^g \cdot\varepsilon_{2p}^g\\
	&=&	 (-1)^{a+b+ c+d +g  }\varepsilon_{  q}^c\varepsilon_{2pq}^f \varepsilon_{2p}^g  .
\end{eqnarray*}
Thus   $a+b+ c+d +g      =0\pmod 2$ and $c+f+  g      =0\pmod 2$.   

\noindent\ding{224} Let us apply the norm $N_{\KK/k_6}=1+\tau_1\tau_3$, with $k_6=\QQ(\sqrt{p}, \sqrt{2q})$. By the table \eqref{norm q=7 p=1mod4q} and Table \ref{tab3 case: 7=1 mod 8 and ()=1}, we  have:
\begin{eqnarray*}
	N_{\KK/k_6}(\xi^2)&=&(-1)^a\cdot \varepsilon_{p}^{2b} \cdot (-1)^c  \cdot (-1)^d\cdot \varepsilon_{2  q}^d \cdot   \varepsilon_{2pq}^f\cdot (-1)^{gu+g}\\
	&=&	   \varepsilon_{p}^{2b}\varepsilon_{2pq}^f (-1)^{a+c+d +gu+g}\varepsilon_{2  q}^d.
\end{eqnarray*}	
Thus $d=0$ and so $a+c  +gu+g =0\pmod 2$.   
Therefore,
$$\xi^2=\varepsilon_{2}^a\varepsilon_{p}^b \sqrt{\varepsilon_{q}}^c \sqrt{\varepsilon_{pq}\varepsilon_{ 2pq}}^f \sqrt{ \varepsilon_{2p}}^g.$$

\noindent\ding{224} Let us apply the norm $N_{\KK/k_4}=1+\tau_1$, with $k_4=\QQ(\sqrt{p}, \sqrt{q})$.  By the table \eqref{norm q=7 p=1mod4q} and Table \ref{tab3 case: 7=1 mod 8 and ()=1}, we  have:		
\begin{eqnarray*}
	N_{\KK/k_4}(\xi^2)&=&(-1)^{a}\cdot \varepsilon_{p}^{2a}\cdot (-1)^c\cdot \varepsilon_{q}^{c}  \cdot \varepsilon_{  pq}^f\cdot (-1)^{gu+g}\\
	&=& \varepsilon_{p}^{2b}\varepsilon_{  pq}^f(-1)^{a+c +gu+g}\varepsilon_{q}^{c} .
\end{eqnarray*}
Thus       $c=0$ and so  $a  +gu+g=0\pmod 2$. Since $c+f+  g      =0\pmod 2$, we have $f=g$. 
 Therefore,
 $$\xi^2=\varepsilon_{2}^a\varepsilon_{p}^{gu}   \sqrt{\varepsilon_{pq}\varepsilon_{ 2pq}}^g \sqrt{ \varepsilon_{2p}}^g,$$
 with $a  +gu+g=0\pmod 2$. So we have the second item.

		\end{enumerate}	
\end{proof}

		\begin{exam}Using  PARI/GP calculator version  2.15.5 (64bit), Feb 11 2024, we  get the following examples which illustrate the above theorem.
		\begin{enumerate}[\rm $1)$]
			\item   Consider the prime numbers $p=313$ and  $q=7$. In this case, the   conditions of the first item of the above theorem are verified and we have:
			$$E_{\KK}=\langle -1,   \varepsilon_{2}, \varepsilon_{p} ,     \sqrt{\varepsilon_{ q}}, \sqrt{\varepsilon_{2q}},\sqrt{ \varepsilon_{pq}} ,\sqrt{ \varepsilon_{2pq}}, \sqrt{\varepsilon_{2}\varepsilon_{p}\varepsilon_{2p}}   \rangle.$$
			Furthermore, we have $h_2(2p)=8$, $h_2(pq)=4$,  $h_2(2pq)=8$ and  $h_2(\KK)=8$.
			\item   Consider the prime numbers $p=73$ and  $q=383$. In this case, the   conditions of the second item of the above theorem are verified and we have:
			$$E_{\KK}=\langle -1,   \varepsilon_{2}, \varepsilon_{p},  \sqrt{\varepsilon_{q}}, \sqrt{\varepsilon_{2q}},  \sqrt{\varepsilon_{pq}} ,\sqrt{ \varepsilon_{2pq}},      \sqrt{     \varepsilon_{2}^{a}\varepsilon_{p}^{ u}   \sqrt{\varepsilon_{pq}\varepsilon_{ 2pq}} \sqrt{ \varepsilon_{2p}}   }  \rangle,$$
			for some 	$a\in\{0,1\}$ such that $a \equiv 1+u\pmod2$.	Furthermore, we have $h_2(2p)=2$, $h_2(pq)=8$,  $h_2(2pq)=4$ and  $h_2(\KK)=8$.
		\end{enumerate}
	\end{exam}

 \section*{Acknowledgment}
 We would like to thank the unknown referee  for his/her several helpful suggestions that helped us to improve our paper.
% \footnotesize

\end{document}